\documentclass[12pt,reqno]{amsart}
\usepackage{hyperref}
\usepackage{color, bm, amscd, tikz-cd}
\newcommand{\cB}{{\mathcal B}}

\newcommand{\cD}{{\mathcal D}}
\newcommand{\cE}{{\mathcal E}}
\newcommand{\cF}{{\mathcal F}}

\newcommand{\cH}{{\mathcal H}}

\newcommand{\cL}{{\mathcal L}}

\newcommand{\cN}{{\mathcal N}}

\newcommand{\cP}{{\mathcal P}}

\newcommand{\cR}{{\mathcal R}}

\newcommand{\cU}{{\mathcal U}}

\newcommand{\cW}{{\mathcal W}}

\newcommand{\sbm}[1]{\left[\begin{smallmatrix} #1
		\end{smallmatrix}\right]}

\newtheorem{thm}{Theorem}[section]

\newtheorem{corollary}[thm]{Corollary}
\newtheorem{lemma}[thm]{Lemma}

\newtheorem{proposition}[thm]{Proposition}

\theoremstyle{definition}

\newtheorem{definition}[thm]{Definition}
\newtheorem{remark}[thm]{Remark}
\newtheorem{example}[thm]{Example}

\numberwithin{equation}{section}

\def\textmatrix#1&#2\\#3&#4\\{\bigl({#1 \atop #3}\ {#2 \atop #4}\bigr)}
\def\dispmatrix#1&#2\\#3&#4\\{\left({#1 \atop #3}\ {#2 \atop #4}\right)}
\numberwithin{equation}{section}

\def\textmatrix#1&#2\\#3&#4\\{\bigl({#1 \atop #3}\ {#2 \atop #4}\bigr)}
\def\dispmatrix#1&#2\\#3&#4\\{\left({#1 \atop #3}\ {#2 \atop #4}\right)}

\begin{document}

\title[Distinguished varieties]{Distinguished varieties through the Berger--Coburn--Lebow Theorem}

\author[T. Bhattacharyya]{Tirthankar Bhattacharyya}
\address{Department of Mathematics, Indian Institute of Science, Bengaluru-560012, India}
\email{tirtha@iisc.ac.in}

\author[P. Kumar]{Poornendu Kumar}
\address{Department of Mathematics, Indian Institute of Science, Bengaluru-560012, India}
\email{poornendukumar@gmail.com}

\author[H. Sau]{Haripada Sau}
\address{Tata Institute of Fundamental Research, Centre for Applicable Mathematics, Sharada Nagar, Chikkabommsandra, Bengaluru-560065, India.}
\email{haripadasau215@gmail.com}

\subjclass[2010]{Primary: 47A13. Secondary: 32C25, 47A20, 47A48, 47A57.}
\keywords{Distinguished varieties, commuting isometries, inner functions, linear pencils}

\maketitle

\begin{abstract}

A distinguished algebraic variety in $\mathbb{C}^2$ has been the focus of much research in recent years because of good reasons. This note gives a different perspective.

\begin{enumerate}

\item We find a new characterization of an algebraic variety $\cW$ which is distinguished with respect to the bidisc. It is in terms of the joint spectrum of a pair of commuting linear matrix pencils.

\item There is a characterization known of $\mathbb{D}^2\cap\mathcal{W}$ due to a seminal work of Agler and McCarthy. We show that Agler--McCarthy characterization can be obtained from the new one and vice versa.

\item En route, we develop a new realization formula for operator-valued contractive analytic functions on the unit disc.

\item There is a one-to-one correspondence between operator-valued contractive holomorphic functions and {\em canonical model triples}. This pertains to the new realization formula mentioned above.

\item Pal and Shalit gave a characterization of an algebraic variety, which is distinguished with respect to the symmetrized bidisc, in terms of a matrix of numerical radius no larger than $1$. We refine their result by making the class of matrices strictly smaller.

\item In a generalization in the direction of more than two variables, we characterize all one-dimensional algebraic  varieties which are distinguished  with respect to the polydisc.
\end{enumerate}
At the root of our work is the Berger--Coburn--Lebow theorem characterizing a commuting tuple of isometries.

\end{abstract}

\section{Introduction}
One of the central objects in algebraic geometry is an algebraic variety.

\begin{definition}
A subset $\mathcal W$ of the $d$-dimensional complex Euclidean space $\mathbb{C}^d$ is called an {\em algebraic variety} if
$$\mathcal{W}=\{(z_1, z_2, \ldots,z_d)\in{\mathbb{C}}^d: \xi_{\alpha}(z_1, z_2, \ldots,  z_d)=0 \text{ for all } \alpha\in{\Lambda}\}$$
where $\Lambda$ is an index set and $\xi_{\alpha}$ are in $\mathcal{C}[z_1, z_2, \ldots, z_d]$, the ring of polynomials in $d$ variables with complex coefficients.\end{definition}

In other words, $\mathcal{W}$ is the {\em zero set} $Z(S)$ of the collection of  polynomials $S=\{\xi_{\alpha}: \alpha\in{\Lambda}\}$. The objects of study in this note are certain special algebraic varieties defined below.
\begin{definition}
Given a non-empty, polynomially convex domain $\Omega$ in $\mathbb{C}^d$, its {\em distinguished boundary} $b\Omega$ is defined to be the smallest closed subset $C$ of $\overline{\Omega}$ such that every function in $\mathcal{A}(\Omega)$, the algebra of continuous functions on $\overline{\Omega}$ which are holomorphic in $\Omega$, attains its maximum modulus on $C$.

Let $\partial \Omega$ be the topological boundary of $\Omega$. Given an algebraic variety $\mathcal{W}$ in $\mathbb C^d$ if it so happens that $\mathcal{W} \cap \partial{\Omega} = \mathcal{W} \cap {b\Omega}$ and $\cW\cap\Omega$ is non-empty, then $\mathcal W$ is called a {\em distinguished variety} with respect to $\Omega$.
\end{definition}

         Given a non-trivial algebraic variety $\mathcal{W}$ in $\mathbb{C}^2$, by a general result, we can find one polynomial $\xi$ such that $\mathcal{W}=Z(\xi)$, the zero set of $\xi$. A large part of this note will be concerned with distinguished varieties with respect to the bidisc, i.e., those $Z(\xi)$ which satisfy $Z(\xi)\cap\mathbb{D}^2\neq\emptyset$ and $Z(\xi)\cap\partial\mathbb{D}^2 = Z(\xi)\cap\mathbb{T}^2$, where $\mathbb{D}=\{z\in\mathbb{C}: |z|<1\}$, the open unit disc in $\mathbb{C}$. Sometimes we shall call it just a {\em distinguished variety} while it will be clear from the the context that we are dealing with a distinguished variety with respect to the bidisc.

         Distinguished varieties have been in the focus ever since it was proved that for a pair of commuting matrices $T_1$ and $T_2$ which satisfy  \begin{enumerate}
\item $\|T_i\|\leq 1 $ for $i=1, 2,$
\item neither $T_1$ nor $T_2$ has an eigenvalue of modulus $1$,\end{enumerate} there is a distinguished variety $\mathcal{W}$ such that
$$
\|\xi(T_1,T_2)\|\leq \sup_{\mathbb D^2\cap\mathcal{W}}|\xi(z_1, z_2|,
$$ for every $\xi\in\mathbb{C}[z_1, z_2]$. A considerable amount of analysis has been carried out on distinguished varieties, see \cite{AM}, \cite{AM2006}, \cite{DGH}, \cite{DSJFA}, \cite{DSS}, \cite{DJM}, \cite{DU}, \cite{JKM}, \cite{Knese2007}, \cite{Knese-TAMS2010}, \cite{DS} and \cite{Vegulla}. Because of such prominence of distinguished varieties, a simple description is desirable. If one looks at $\mathbb{D}^2\cap\mathcal{W}$, then Agler and McCarthy gave a characterization in Theorem 1.12 of \cite{AM}. We improve their result in Section \ref{throughinner}.

       Due to the presence of a rational function in the characterization obtained in \cite{AM}, it is inherently challenging to give the description of the whole variety $\mathcal W$ instead of just $\mathbb{D}^2\cap\mathcal{W}$. Motivated by this, we discover a commuting pair of linear matrix pencils whose joint spectrum catches the whole of $\mathcal{W}$. This is our main theorem and is the content of Section \ref{S:LinPencil}.

       How to go back and forth between these two descriptions of a distinguished variety now becomes a natural question. This question leads us to a new realization formula for a bounded holomorphic operator-valued function on the unit disc. This new realization formula is different from the classically well-known one and allows us to tie up our new description of a distinguished variety with that of Agler and McCarthy. This is the content of Section \ref{AMandBCL}.

Often in this note we shall talk of vector-valued Hardy spaces. For a Hilbert space $\mathcal E$, the Hardy space of $\mathcal E$-valued functions is
$$H^2(\mathcal E) = \{ f : \mathbb D \rightarrow \mathcal E \mid f \text{ is analytic and } f(z) = \sum_{n=0}^{\infty} a_n z^n \text{ with } \sum_{n=0}^{\infty} \|a_n\|_\cE^2 < \infty\}.$$
Here the $a_n$ are from $\mathcal E$. This is a Hilbert space with the inner product
$$\langle \sum_{n=0}^{\infty} a_n z^n , \sum_{n=0}^{\infty} b_n z^n \rangle =  \sum_{n=0}^{\infty} \langle a_n , b_n \rangle_\cE$$ and is identifiable with $H^2 \otimes \mathcal E$ where $H^2$ stands for the Hardy pace of scalar-valued functions on $\mathbb D$. Naturally, if $\varphi$ is a $\mathcal B (\mathcal E)$-valued bounded analytic function on $\mathbb D$, then it induces a multiplication operator $M_\varphi$ on $H^2(\mathcal E)$.

       The new developments, viz., the main theorem characterizing distinguished varieties and the realization formula depend on a crucial result of Berger, Coburn and Lebow (Theorem 3.1 in \cite{BCL}) which states that given a pair of commuting isometries $(V_1, V_2)$ on a Hilbert space $\mathcal{H}$, there is a reducing subspace $\cH_u$ (reducing both $V_1$ and $V_2$) such that the product $V=V_1V_2$, when restricted to $\cH_u$, is a unitary operator and the pair $(V_1|_{\cH_u^\perp}, V_2|_{\cH_u^\perp})$ is unitarily equivalent to a commuting pair of multiplication operators  $(M_{P^\perp U+zPU}, M_{U^*P+zU^*P^\perp})$ where $U$ and $P$ are respectively a unitary operator and a projection operator acting on some Hilbert space $\mathcal{F}$ and the multiplication operators act on the vector-valued Hardy space $H^2(\mathcal{F})$. In fact, $\cH_u$ is the unitary part in the wold decomposition of the isometry $V$. We shall refer to this result as the {\em Berger--Coburn--Lebow (BCL)  Theorem}.

       Because of the BCL Theorem, the triple $(\cF, P, U)$ will have a large presence in our work. So we make a definition.

 \begin{definition}
 If $P$ is an orthogonal projection and $U$ is a unitary acting on a Hilbert space $\cF$, then we shall call $(\cF, P, U)$ a {\em model triple}. When $\cF$ is finite dimensional, $(\cF, P, U)$ will be called a {\em finite (dimensional) model triple}.

 Two model triples $(\cF_1, P_1, U_1)$ and $(\cF_2, P_2, U_2)$ are called {\em unitarily equivalent} if there is a unitary operator from the Hilbert space $\cF_1$ to the Hilbert space $\cF_2$ which intertwines the projection $P_1$ with the projection $P_2$ and intertwines the unitary $U_1$ with the unitary $U_2$.
 \end{definition}

\begin{definition}
The pair of commuting isometries
\begin{align}\label{BCLmodel}
(M_{(P^\perp  + zP)U},M_{U^*(P + zP^\perp)})
\end{align}will be called the {\em Berger--Coburn--Lebow model} (\textit{BCL model}) associated to the model triple $(\cF,P,U)$.

Given a model triple, the two functions $P^\perp U+zPU$ and $U^*P + zU^*P^\perp$ will be called the {\em Berger--Coburn--Lebow (BCL) functions}.
\end{definition}

           Every model triple gives a contractive analytic function. The realization formula mentioned above proves that every contractive analytic function arises from a model triple.

           More can be said on the relationship of model triples and contractive analytic functions. It is easy to see that unitarily equivalent model triples give rise to unitarily equivalent contractive analytic functions. The converse of this statement cannot be formulated without a further intricate analysis. We show that given a contractive analytic function, a canonical choice of a model triple can be naturally made so that when two contractive analytic functions are unitarily equivalent, so are the associated canonical model triples. This category theoretic result is in Section \ref{Category} .

              The new characterization of distinguished varieties with respect to the bidisc influences the characterization of distinguished varieties with respect to the symmetrized bidisc
$$ \mathbb G := \{ (z_1 + z_2, z_1z_2) : |z_1| < 1 \mbox{ and } |z_2| < 1\}.$$
Pal and Shalit characterized all distinguished varieties with respect to the symmetrized bidisc in \cite{Pal-Shalit}. We give an improvement of this result, as an application of the main theorem, in Section \ref{S:SymmBDisc}.

           In the final section, we characterize all one-dimensional distinguished varieties of the polydisc using the full force of Theorem 3.1 of \cite{BCL} by Berger, Coburn and Lebow.

\section{Through the inner function} \label{throughinner}

A matrix-valued holomorphic function $\Psi$ on $\mathbb{D}$ is said to be {\em rational inner} if there is a block unitary matrix
$\cU=\begin{bmatrix}
A & B \\
C & D
\end{bmatrix}$
such that $\Psi(z)=A+zB(I-zD)^{-1}C$. It can be assumed without loss of generality that the contractive matrix $D$ has no unimodular eigenvalue. Thus, $\Psi$ is analytic in a neighbourhood of $\overline{\mathbb D}$ and $\Psi(z)$ is a unitary matrix if $z\in\mathbb{T}$. Obviously, there is a matrix polynomial $F$ and scalar polynomial $q$ such that $\Psi(z)=F(z)/q(z)$ and there is no factor of $q$ that divides every entry of $F$. The poles of $\Psi$, i.e., the zeros of $q$, are away from $\overline{\mathbb D}$.

If $\cW_{\Psi}$ is the algebraic variety
$$\cW_{\Psi}:=\{(z_1, z_2)\in\mathbb{C}^2: \det(F(z_1)-z_2q(z_1)I)=0\},$$
then $\mathbb{D}^2\cap\cW_{\Psi}=\{(z_1, z_2)\in{\mathbb{D}^2}: \det(\Psi(z_1)-z_2I)=0\}$.

     Let $\alpha$ be a zero of $q$. If there is a $\beta$ such that $(\alpha, \beta)\in{\cW_{\Psi}}$, then $\det F(\alpha)=0$. This means that $(\alpha , z_2)$ is a zero of the polynomial $\det(F(z_1)-z_2q(z_1)I)$ for every $z_2$. Thus, there is an $m_{\alpha} \ge 1$ such that $\det(F(z_1)-z_2q(z_1)I)$ is divisible by $(z_1-\alpha)^{m_\alpha} $. Take the largest such $m_\alpha$ for every $\alpha$ that is a zero of $q$. Then, there is a polynomial $\xi_{\Psi}$ such that
\begin{align}\label{XiPsi}
 \det(F(z_1)-z_2q(z_1)I)=\prod_{\alpha\in Z(q)}(z_1-\alpha)^{m_\alpha}\xi_\Psi(z_1,z_2)
 \end{align}
 with the understanding that $m_\alpha$ could be $0$ for some $\alpha$ (precisely those $\alpha$ for which there is no $\beta$ satisfying $(\alpha, \beta)\in{\cW_{\Psi}}$). We now restate the Agler--McCarthy--Knese theorem in a way that allows a description of the whole variety instead of just its portion in $\mathbb{D}^2$, under a natural condition.

 We use the notations $\mathbb E$ for the complement of $\overline{\mathbb D}$ in $\mathbb C$, and $\nu(A)$ for the {\em numerical radius} of a square matrix $A$.

  \begin{thm}\label{Sec-AMDVTheorem}
Let $\Psi$ and $\cW_\Psi$ be as above. Then the following are equivalent:
\begin{itemize}
\item[(i)] $\cW_\Psi$ is a distinguished variety with respect to $\mathbb D^2$;
\item[(ii)] $\nu(\Psi(z)) < 1$ for all $z$ in $\mathbb D$;
\item[(iii)] $Z(\xi_\Psi)\subset \mathbb D^2\cup\mathbb T^2\cup\mathbb E^2$;
\item[(iv)] $Z(\xi_\Psi)$ is a distinguished variety.
\end{itemize}

Conversely, if $\cW$ is any distinguished variety with respect to $\mathbb D^2$, then there is a matrix-valued rational inner function $\Psi$ on $\mathbb D$ such that
$$
\cW\cap\mathbb D^2=\cW_\Psi\cap\mathbb D^2=Z(\xi_\Psi)\cap\mathbb D^2.
$$Moreover, $\cW=Z(\xi_\Psi)$ if and only if both $\cW$ and $Z(\xi_\Psi)$ are contained in $\mathbb D^2\cup\mathbb T^2\cup\mathbb E^2$.
\end{thm}

Before we prove this theorem, we pause to note that a distinguished variety need not be contained in $\mathbb D^2\cup\mathbb T^2\cup\mathbb E^2$, in general. For example, consider the variety $Z(\xi)$, where $\xi(z_1,z_2)=(z_1-z_2)(z_1z_2-1)$. Then
$$Z(\xi)\cap\mathbb D^2= \{(z,z):z\in\mathbb D\},$$
the diagonal variety and $Z(\xi)\cap\partial \mathbb D^2=Z(\xi)\cap\mathbb T^2=\{(z,z):z\in\mathbb T\}\cup\{(z,\bar z):z\in\mathbb T\}$. Therefore $Z(\xi)$ is a distinguished variety with respect to $\mathbb D^2$, but it is clearly not contained in $\mathbb D^2\cup\mathbb T^2\cup\mathbb E^2$. However, Knese \cite[Proposition 4.1]{Knese-TAMS2010} showed that if a variety $Z(\xi)$ is such that each of its irreducible components intersects $\mathbb D^2$, then $Z(\xi)$ must be contained in $\mathbb D^2\cup\mathbb T^2\cup\mathbb E^2$. The equivalence of (iii) and (iv) in Theorem \ref{Sec-AMDVTheorem} infers that the only way the variety $Z(\xi_\Psi)$ can be distinguished is that it be contained in $\mathbb D^2\cup\mathbb T^2\cup\mathbb E^2$. We shall demonstrate one more class of distinguished varieties adhering to this phenomenon in Theorem \ref{ourDV}.
\begin{proof}We complete the proof by establishing
$$
(ii)\Leftrightarrow(i)\Leftrightarrow(iv)\Leftrightarrow(iii).
$$
{\em Proof of $(ii)\Leftrightarrow(i)$:} First we assume (ii). To show that $\cW_\Psi$ is a distinguished variety, we first prove that $\cW_\Psi\cap\mathbb D^2$ is non-empty. Let $z_1\in\mathbb D$. Then by (ii), any eigenvalue $z_2$ of $\Psi(z_1)$ will come from $\mathbb D$ and for any such $z_2$, $(z_1,z_2)\in\cW_\Psi$. Secondly, we show that $\cW_\Psi\cap\partial\mathbb D^2=\cW_\Psi\cap\mathbb T^2$. To that end let $(z_1,z_2)\in\cW_\Psi\cap\partial\mathbb D^2$. If $|z_1|<1$, then by (ii) again $|z_2|<1$ . If $|z_1|=1$, then since $\Psi$ is rational inner, $\Psi(z_1)$ is unitary and hence $|z_2|=1$. Therefore $\cW_\Psi$ is a distinguished variety.

 Conversely, assume (i). Suppose on the contrary that there is a point $z_0$ in $\mathbb D$ such that $\nu(\Psi(z_0)) = 1$. Then there is a unit vector $h_0$ such that $|\langle \Psi(z_0) h_0 , h_0\rangle | = 1$ and this, through an application of Cauchy-Schwarz inequality, gives us $\Psi(z_0)h_0 = \exp{(i\theta)} h_0$ for some $\theta$ showing that the point $( z_0, \exp{(i\theta)})$ is in $\cW_\Psi$. This violates the distinguishedness of $\cW_{\Psi}$.

{\em Proof of $(i)\Leftrightarrow(iv)$:} Note that by \eqref{XiPsi}, $\cW_\Psi$ is exactly the union of $Z(\xi_\Psi)$ and the complex lines $\{\alpha\}\times\mathbb C$, where $\alpha\in Z(q)$ is such that $m_\alpha$ as in \eqref{XiPsi} is non-zero. Since the lines do not intersect $\mathbb D^2$, it is trivial that $\cW_\Psi$ is a distinguished variety if and only if $Z(\xi_\Psi)$ is so.

{\em Proof of $(iii)\Leftrightarrow(iv)$:} The implication $(iii)\Rightarrow(iv)$ is obvious. For the other direction, we invoke the following result of Knese.
\begin{proposition}[Proposition 4.1 in Knese \cite{Knese-TAMS2010}]
An algebraic variety $\cW$ each of whose irreducible components intersects $\mathbb D^2$ is distinguished with respect to $\mathbb D^2$ if and only if
$$
\cW\subset \mathbb D^2\cup\mathbb T^2\cup\mathbb E^2.
$$
\end{proposition}
To show that each irreducible component of $Z(\xi_\Psi)$ intersects $\mathbb D^2$, we assume with no loss of generality that $\xi_\Psi$ is irreducible. Let us view $\xi_\Psi(z_1,z_2)$ as a polynomial in $z_2$ for each fixed $z_1$. There exists at least one $\alpha\in\mathbb D$, for which $\xi_\Psi(\alpha,z)$ is a non-constant polynomial in $z$ because otherwise $\xi_\Psi(z_1,z_2)$ would be a polynomial in $z_1$ only and that would violate the distinguishedness of $Z(\xi_\Psi)$. Let $\beta$ be a root of $\xi_\Psi(\alpha,z)$. Then $(\alpha,\beta)\in\cW_\Psi$. This implies that $\beta$ is an eigenvalue of $\Psi(\alpha)$. Since $(iv)\Rightarrow(ii)$ is already established and $\alpha$ is in $\mathbb D$, we use (ii) to conclude that $\beta\in\mathbb D$ as well. This completes the proof of the forward direction of the theorem.

    Conversely, if $\cW$ is a distinguished variety with respect to $\mathbb{D}^2$, then by  \cite[Theorem 1.12]{AM}, there exists a matrix-valued rational inner function $\Psi$ on $\mathbb D$ such that $\cW\cap\mathbb D^2=\cW_\Psi\cap\mathbb D^2$.

For the moreover part, suppose first that $\cW=Z(\xi_\Psi)$. Since $\cW$ is a distinguished variety, by the equivalence we saw above, $Z(\xi_\Psi)\subset\mathbb D^2\cup\mathbb T^2\cup\mathbb E^2$. Conversely, let both $\cW$ and  $Z(\xi_{\Psi})$ be contained in $\mathbb D^2\cup\mathbb T^2\cup\mathbb E^2$. Now, we need an idea of Knese.

\begin{definition}
A polynomial $\xi$ with highest powers $m_1$ and $m_2$ of $z_1$ and $z_2$ respectively, is called {\em essentially $\mathbb{T}^2$-symmetric} if
$$\xi(z_1, z_2)=cz_1^{m_1}z_2^{m_2}\overline{\xi(\frac{1}{\overline{z_1}}, \frac{1}{\overline{z_2}})}$$
for some unimodular constant $c$.
\end{definition}

We state a proposition from Knese \cite{Knese-TAMS2010}.

\begin{proposition}(Proposition 4.3 of Knese \cite{Knese-TAMS2010})
A polynomial $\xi$ is essentially $\mathbb{T}^2$-symmetric if $Z(\xi)\subset\mathbb D^2\cup\mathbb T^2\cup\mathbb E^2$.

\end{proposition}

In view of this proposition, both $\cW$ and $Z(\xi_{\Psi})$ have the symmetry property:

\begin{definition}
A variety $\cW$ is said to be {\em symmetric} if for any $z_1\neq 0$ and $z_2\neq 0$, we have
\begin{equation}\label{symm}
(z_1, z_2)\in\cW \text { if and only if } (\frac{1}{\overline{z_1}}, \frac{1}{\overline{z_2}})\in\cW.
\end{equation}
\end{definition}

Thus, $\cW\cap\mathbb{E}^2$ is determined by $\cW\cap\mathbb{D}^2$ and the same holds for  $Z(\xi_\Psi)$. So  $\cW\cap(\mathbb{D}^2\cup\mathbb{E}^2)=  Z(\xi_\Psi)\cap(\mathbb{D}^2\cup\mathbb{E}^2)$. Now, note that every irreducible component of $\cW$ meets the bidisc and the same is true for  $Z(\xi_\Psi)$. So, for a point $(z_1, z_2)\in\mathbb{T}^2\cup\cW$, an approximation argument shows that $(z_1, z_2)\in\mathbb{T}^2\cup Z(\xi_\Psi)$ and vice versa. This completes the proof of the theorem.

\end{proof}

\section{Through the linear pencils}\label{S:LinPencil}

\subsection{The joint spetrum}

The beginning of this section warrants a discussion of the joint spectrum.
The joint spectrum $\sigma_T(\underline{T})$ of a commuting tuple ${\underline{T}}=(T_1,T_2, \ldots ,T_d)$ of operators was defined by Taylor in \cite{Taylor} using a Koszul complex. When the Hilbert space is of finite-dimension, say $n$, the joint spectrum is particularly simple to describe because a set of commuting matrices can be simultaneously upper-triangularized by a unitary matrix. This means that one can choose an orthonormal basis of the Hilbert space with respect to which the linear transformations $T_i$ are of the form
  \[
    \begin{bmatrix}
       \lambda_{1}^{(i)}& & & *  \\
       & \lambda_{2}^{(i)} & & &  \\
       & & \ddots & & \\
       0 & & & \lambda_{n}^{(i)}\\
    \end{bmatrix}.
\]
          The joint spectrum $\sigma_T(T)$ then consists of $\{(\lambda_{j}^{(1)},\lambda_{j}^{(2)}, \ldots ,\lambda_{j}^{(d)}):j=1,2, \ldots ,n\}$. Indeed, every $(\lambda_{j}^{(1)},\lambda_{j}^{(2)}, \ldots ,\lambda_{j}^{(d)})$ is a {\em joint eigenvalue}, i.e., there is a common non-zero eigenvector $x_j$ that satisfies $T_i x_j = \lambda_{j}^{(i)}x_j$ for $i=1,2, \ldots , d$. This is analogous to the fact that the spectrum of a linear transformation on a finite dimensional space consists of just eigenvalues.

         We start with a preparatory lemma on joint spectra.

\begin{lemma} \label{jemultipliers}
If $\mathcal F$ is a finite dimensional Hilbert space and if $\Phi$ and $\Psi$ are two $B(\mathcal{F})$-valued bounded holomorphic functions on $\mathbb{D}$ satisfying $\Phi(z)\Psi(z)=zI_{\mathcal{F}}=\Psi(z)\Phi(z)$ for every $z\in{\mathbb{D}}$, then $(\overline{z_1},\overline{z_2})\in{\mathbb{D}^2}$ is a joint eigenvalue of $(M_{\Phi}^*, M_{\Psi}^*)$ if and only if $(z_1, z_2)$ is a joint eigenvalue of $(\Phi(z_1z_2), \Psi(z_1z_2))$.
\end{lemma}
\begin{proof}
 Let $(z_1, z_2)\in{\mathbb{D}^2}$ be a joint eigenvalue of $(\Phi(z_1z_2), \Psi(z_1z_2))$. It is easy to see that $(\overline{z_1}, \overline{z_2})$ is a joint eigenvalue of $(\Phi(z_1z_2)^*, \Psi(z_1z_2)^*)$. Let $w$ be a joint eigenvector for $(\Phi(z_1z_2)^*, \Psi(z_1z_2)^*)$ corresponding to the joint eigenvalue $(\overline{z_1}, \overline{z_2})$. Then
$$M_{\Phi}^*(k_{z_1z_2}\otimes w)=k_{z_1z_2}\otimes\Phi(z_1z_2)^*w=k_{z_1z_2}\otimes\overline{z_1}w =\overline{z_1}(k_{z_1z_2}\otimes w).$$

 Similarly, $M_{\Psi}^*(k_{z_1z_2}\otimes w)=\overline{z_2}(k_{z_1z_2}\otimes w)$. Thus,
$(\overline{z_1},\overline{z_2})$ is a joint eigenvalue of $(M_{\Phi}^*, M_{\Psi}^*)$.

If $(\overline{z_1},\overline{z_2})\in{\mathbb{D}^2}$ is a joint eigenvalue of $(M_{\Phi}^*, M_{\Psi}^*)$, let $f(z)=\sum_{n=0}^{\infty}\zeta_n z^n\in{H^2(\mathcal{F})}$ be a joint eigenvector. For any $g\in{H^2(\mathcal{F})}$, we have
$$\overline{z_1} \overline{z_2}\langle f, g\rangle = \langle \overline{z_1} \overline{z_2} f, g\rangle = \langle M_{\Phi}^*M_{\Psi}^* f, g\rangle
= \langle (M_{\Phi}M_{\Psi})^*f, g\rangle
=\langle M_z^*f, g \rangle
=\langle f, M_zg\rangle.$$
 Taking $g(z)=z^{n-1}\zeta$ for $n=1,2, \ldots, $ we have
 $$\overline{z_1} \overline{z_2}\langle \zeta_{n - 1}, \zeta\rangle=\langle \zeta_n, \zeta\rangle.$$
Since  $\zeta$ is arbitrary, we have $\overline{z_1}\overline{z_2}\zeta_{n-1}=\zeta_n$ for $n=1,2, \ldots $. So, $\zeta_n= (\overline{z_1}\overline{z_2})^n\zeta_{0} $ for $n= 1, 2, \ldots$.
 So, $f(z)=k_{z_1z_2}\otimes \zeta_{0}$. So,
 \begin{align}\label{Multi_equ}
 M_{\Phi}^* f= M_{\Phi}^*(k_{z_1z_2}\otimes\zeta_{0})=k_{z_1z_2}\otimes\Phi(z_1z_2)^*\zeta_{0}.
 \end{align}
  So, $k_{z_1z_2}\otimes \Phi(z_1z_2)^*\zeta_{0}=k_{z_1z_2}\otimes\overline{z}_1\zeta_{0}$. Equating the constant terms on both sides, we get
  $$\Phi(z_1z_2)^*\zeta_{0}=\overline{z_1}\zeta_{0}.$$
  Similarly, using $\Psi$ in place of $\Phi$ in \eqref{Multi_equ}, we have $\Psi(z_1z_2)^*\zeta_{0}=\overline{z_2}\zeta_{0}$. Thus, $(\overline{z_1}, \overline{z_2})$ is a joint eigenvalue of the pair of commuting matrices $(\Phi(z_1z_2)^*, \Psi(z_1z_2)^*)$ which in turn implies that $(z_1, z_2)$ is a joint eigenvalue of $(\Phi(z_1z_2), \Psi(z_1z_2))$.

\end{proof}

Given a finite model triple $(\cF,P,U)$, the BCL functions $P^\perp U+zPU$ and $U^*P + zU^*P^\perp$ satisfy the required condition above, i.e.,
$$(P^\perp U+zPU) (U^*P + zU^*P^\perp) = zI_{\mathcal F}=(U^*P + zU^*P^\perp)(P^\perp U+zPU)$$ for all $z$ in $\mathbb D$.

This has an interesting consequence: $(z_1, z_2) \in \sigma_T(P^\perp U+z PU,U^*P+z U^*P^\perp)$ implies that $z = z_1z_2$.

\subsection{The Main Theorem}\label{SS:TheMainThm}
At the end of the last section, we saw the importance of the symmetry property. In view of Knese's work \cite[Proposition 4.1]{Knese-TAMS2010}, a symmetric distinguished variety $\cW$ all whose components intersect the bidisc is determined by $\cW \cap \mathbb D^2$. Thus, it is natural to ask for a description of a distinguished variety which is automatically symmetric.

This consideration prompts us to the main theorem.

\begin{thm}\label{ourDV}
For an orthogonal projection $P$ and a unitary $U$ acting on a finite dimensional Hilbert space, the set
$$
\cW_{P,U}:=\{(z_1,z_2)\in\mathbb{C}^2:(z_1,z_2)\in\sigma_T(P^\perp U+z_1z_2PU,U^*P+z_1z_2U^*P^\perp)\}
$$
is a symmetric algebraic variety in $\mathbb C^2$ for which the following are equivalent:
\begin{itemize}
\item[(i)] $\cW_{P,U}$ is a distinguished variety with respect to $\mathbb D^2$;
\item[(ii)]  For all $z$ in the open unit disc $\mathbb D$,
\begin{equation} \label{StrangeConditions1}
 \nu(U^*(P + z P^\perp)) < 1  \text{ and } \nu((P^\perp + zP)U) < 1;
 \end{equation}
\item[(iii)] $\cW_{P,U}\subset \mathbb D^2\cup\mathbb T^2\cup\mathbb E^2$.
\end{itemize}Moreover, $\cW_{P,U}$ can be written as
\begin{equation}\label{Union}
\cW_{P,U}=\bigcup_{z\in\mathbb C}\sigma_T(P^\perp U+zPU,U^*P+zU^*P^\perp).
\end{equation}

Conversely, if $\cW$ is any distinguished variety with respect to $\mathbb D^2$, then there exist an orthogonal projection $P$ and a unitary $U$ acting on a finite dimensional Hilbert space such that
$$
\cW\cap\mathbb D^2=\cW_{P,U}\cap\mathbb D^2.
$$
Moreover, $\cW=\cW_{P,U}$ if and only if both $\cW$ and $\cW_{P,U}$ are contained in $\mathbb D^2\cup\mathbb T^2\cup\mathbb E^2$.
\end{thm}

The condition \eqref{StrangeConditions1} will be referred to as the {\em compatibility condition} for $\mathbb D^2$. We analyze this condition in detail at the end of this section.

For a quick example, note that the finite model triple
\begin{align*} \left(\cF, \begin{bmatrix}
     1 & 0\\
     0& 0
    \end{bmatrix}, \begin{bmatrix}
     0 & 1\\
     1 & 0
    \end{bmatrix}\right)
  \end{align*}
  satisfies the compatibility condition \eqref{StrangeConditions1}, and the associated distinguished variety is $\{ (z,z) : z \in \mathbb C\}$.

An algebraic variety $\cW_{P,U}$ could agree with a distinguished variety inside $\mathbb D^2$ and nevertheless not be a distinguished variety. This happens for
$$
(\cF,P,U)=\left(\mathbb C^4, \begin{bmatrix} P_1&0\\0&P_1\end{bmatrix}, \begin{bmatrix} I_2&0\\0&E_{(12)}\end{bmatrix}\right),
$$
where $P_1=\sbm{1&0\\0&0}$ and $E_{(12)}=\sbm{0&1\\1&0}$. Then
$\cW_{P,U}=\{(z,1),(1,z),(z,z):z\in\mathbb C\}$  is clearly not a distinguished variety, but $\cW_{P,U}\cap\mathbb D^2 = \{(z,z):z\in\mathbb C\}\cap \mathbb D^2$ and $\{(z,z):z\in\mathbb C\}$ is a distinguished variety.

\subsection{Proof of the forward direction}We first prove that $\cW_{P,U}$ is an algebraic variety following a method of Pal \cite{Pal-Tetra}. To that end, consider the family of polynomials
 $$S:=\{\xi_{\lambda, \mu}(z_1,z_2) : \lambda,\mu\in\mathbb C \}$$
 where
   $$ \xi_{\lambda, \mu}(z_1,z_2) := \det\big[\lambda(P^{\perp}U+z_1z_2PU-z_1I)+\mu(U^*P+z_1z_2U^*P^{\perp}-z_2I) ].$$
 We shall prove that $\cW_{P,U}=Z(S)$. Let $(z_1,z_2)\in\cW_{P,U}$. Then there is a vector $h$ of norm $1$ such that
 	 $$(P^{\perp}U+z_1z_2PU)h=z_1h,$$
 	 $$(U^{*}P+z_1z_2U^*P^{\perp})h=z_2h.$$
 Thus for any $\lambda,\mu\in\mathbb C$, we have
 	$$ \big(\lambda(P^{\perp}U+z_1z_2PU-z_1I)+ \mu(U^{*}P+z_1z_2U^*P^{\perp}-z_2I)\big)h=0.$$
 So, $(z_1,z_2)\in{Z(S)}$. Now we prove that $Z(S)\subset\cW_{P,U}$. Let $(z_1,z_2)\in Z(S)$. This means that for any $\lambda,\mu\in\mathbb C$,
 $$\xi_{\lambda,\mu}(z_1,z_2)=\det\big[\lambda(P^{\perp}U+z_1z_2PU-z_1I)+\mu(U^*P+z_1z_2U^*P^{\perp}-z_2I)]=0.$$
 Thus there is a joint eigenvalue, say $(\alpha, \beta)$ (depending on $(z_1,z_2)$), of the commuting pair of matrices
 $$(P^{\perp}U+z_1z_2PU-z_1I , U^*P+z_1z_2U^*P^{\perp}-z_2I)$$
 such that $\lambda \alpha + \mu \beta = 0$ for infinitely many $\lambda$ and $\mu$ in $\mathbb C$. Therefore $(\alpha, \beta) = (0,0)$. This establishes the containment $Z(S)\subset\cW_{P,U}$ and consequently, $\cW_{P,U}=Z(S)$.

The form \eqref{Union} of the algebraic variety $\cW_{P,U}$ is easy to see. Indeed, the product of $P^\perp U+z_1z_2PU$ and $U^*P+z_1z_2U^*P^\perp$ is $z_1z_2I$ implying that any $(\alpha, \beta)$ in $\sigma_T(P^\perp U+z_1z_2PU,U^*P+z_1z_2U^*P^\perp)$ must satisfy $\alpha\beta = z_1z_2$ and hence
 \begin{align*}
 \cW_{P,U} & = \bigcup_{(z_1,z_2)\in\mathbb C^2}\sigma_T(P^\perp U+z_1z_2PU,U^*P+z_1z_2U^*P^\perp) \\
 & = \bigcup_{z\in\mathbb C}\sigma_T(P^\perp U+zPU,U^*P+zU^*P^\perp).
 \end{align*}

To prove symmetry, we use a basic property of joint spectra of matrices. Since, $(z_1, z_2) \in \cW_{P,U}$, we have
$(z_1, z_2)$ is a joint eigenvalue of  $(P^\perp U+z_1z_2PU,U^*P+z_1z_2U^*P^\perp)$. This happens
if and only if $(\frac{1}{\overline{z_1}}, \frac{1}{\overline{z_2}})$ is a joint eigenvalue of $((P^\perp U+z_1z_2PU)^{*^{-1}},(U^*P+z_1z_2U^*P^\perp)^{*^{-1}}) .$
  The proof is complete because
  $$((P^\perp U+z_1z_2PU)^{*^{-1}},(U^*P+z_1z_2U^*P^\perp)^{*^{-1}}) = (P^\perp U+\frac{1}{\overline{z_1z_2}}PU,U^*P+\frac{1}{\overline{z_1z_2}}U^*P^\perp).$$

We establish the equivalence of $(i), (ii)$ and $(iii)$ by showing that
$$
(i)\Rightarrow(ii)\Rightarrow(iii)\Rightarrow(i).
$$

{\em{Proof of $(i)\Rightarrow(ii)$:}} Suppose on the contrary that $\nu(P^\perp U+z_0PU) = 1 $ for some $z_0$ in $\mathbb D$. Then there is an $h_0$ such that $|\langle (P^\perp U+z_0PU)h_0, h_0 \rangle | = 1$ for some unit vector $h_0$. Since $P^\perp U+zPU$ is contractive for every $z \in \mathbb D$, this means, by the condition of equality in Cauchy-Schwarz Inequality, that $(P^\perp U+z_0PU)h_0 = \exp(i \theta)h_0$ for some $\theta$. Thus, $\exp(i \theta)$ is an eigenvalue of $P^\perp U+z_0PU$. So, there is some eigenvalue $\lambda$ of $U^*P+z_0U^*P^\perp$ such that $(\exp(i\theta) , \lambda)$ is a joint eigenvalue of $(P^\perp U+z_0PU, U^*P+z_0U^*P^\perp)$. Since the product of $P^\perp U+z_0PU$ and $U^*P+z_0U^*P^\perp$ is $z_0 I$, we have $\exp(i \theta) \lambda = z_0$. Thus $\lambda$ is in the open disc and this violates distinguishedness of $\cW_{P,U}$. The argument is similar with $U^*P+zU^*P^\perp$ instead of $P^\perp U+zPU$.

{\em{Proof of $(ii)\Rightarrow(iii)$:}} Let $(z_1, z_2) \in \cW_{P,U}$. We complete the proof by considering the following three possibilities:

Case 1: $|z_1z_2| < 1$. The numerical radii conditions force the two eigenvalues $z_1$ and $z_2$ to be in the open unit disc.

Case 2: $|z_1z_2| = 1$. In this case, the matrices $P^\perp U+z_1z_2 PU$ and $U^*P+z_1z_2U^*P$ are unitary matrices. Hence, the two eigenvalues $z_1$ and $z_2$ lie on the unit circle.

Case 3: $|z_1z_2| > 1$. By the symmetric property \eqref{symm},  $(w_1,w_2):=(1/\bar z_1,1/\bar z_2)\in\cW_{P,U}$, and $|w_1w_2| < 1$. Therefore by applying the numerical radii conditions, we get $|w_1|<1$ and $|w_2|<1$. Thus, $|z_1> 1$ and $|z_2| > 1$.

{\em{Proof of $(iii)\Rightarrow(i)$:}} For this part, all we need to establish is the non-emptiness of $\cW_{P,U}\cap\mathbb D^2$. The analysis below shows that actually
$$
\sigma_T(P^\perp U,U^*P)\subset \cW_{P,U}\cap\mathbb D^2.
$$
Since $(P^\perp U , U^*P)$ is a commuting pair of contractive matrices, $\sigma_T(P^\perp U,U^*P)$ is non-empty and contained in $\cW_{P,U}\cap\overline{\mathbb D}^2$. Moreover, since the product is the $0$ matrix, any joint eigenvalue of $(P^\perp U , U^*P)$ is either of the form $(z_1, 0)$ or $(0,z_2)$. Therefore using $(iii)$, we conclude that $\sigma_T(P^\perp U , U^*P) \subset \cW_{P,U}\cap\mathbb D^2$. \qed

It is of independent interest to note that actually each irreducible component of $\cW_{P,U}$ intersects $\mathbb D^2$. This is obvious because the zero set of a non-constant polynomial in $\mathbb C[z_1,z_2]$ cannot be contained in $\mathbb T^2\cup\mathbb E^2$.

\subsection{Proof of the converse part}
Let $\cW = Z(\xi)$ be a distinguished variety with respect to $\mathbb D^2$. With a slight abuse of notation, we denote by $\partial \cW$ the set $Z(\xi) \cap \mathbb T^2$. If $\mu$ is a finite positive measure on $\partial{\cW}$, we denote by $H^2(\mu)$ the norm closure of polynomials in $L^2(\partial{\cW},\mu)$. We shall need a result from the theory of Riemann surfaces.
\begin{lemma}[See Lemma 1.2 of Agler--McCarthy \cite{AM}] \label{AMmeasure}
Let $\cW$ be a distinguished variety with respect to $\mathbb D^2$. There is a positive finite Borel measure on $\partial{\cW}$ such that every point in $\cW\cap\mathbb D^2$ gives rise to a bounded point evaluation for $H^2{(\mu)}$, and such that the span of the bounded evaluation functionals is dense in $H^2{(\mu)}$.
\end{lemma}
Initially, our proof follows the ideas of Agler and McCarthy. Then at one stage, we invoke the theorem by Berger, Coburn and Lebow and find a projection matrix and a unitary matrix to model the concerned pair of pure isometries.

Let $\xi$ be a minimal polynomial in $\mathbb C[z_1,z_2]$ such that
$$\cW=\{(z_1,z_2)\in{\mathbb{C}^2}: \xi(z_1,z_2)=0\}.$$
Let $\mu$ be the measure on $\partial{\cW}$ from Lemma \ref{AMmeasure}. Let $(M_{z_1},M_{z_2})$ be the multiplication operators by the coordinate functions on
 $H^2{(\mu)}$. Since the bounded point evaluation functionals are dense, it follows that $M_{z_1}$, $M_{z_2}$ are pure isometries. Suppose $\xi$ is divisible by $z_1z_2$. Then $\xi(z_1,z_2)=z_1z_2\eta(z_1,z_2)$ for some polynomial $\eta$. Then $(1,0)$ is a point in $\cW$. This is not possible
since $\cW$ is distinguished. Hence, $\xi$ is not divisible by $z_1z_2$. Write
 \begin{align*}
 \xi(z_1,z_2)=\sum_{i=0}^{n}a_iz_1^i+\sum_{j=0}^{m}b_jz_2^{j}+z_1z_2\theta(z_1,z_2)
 \end{align*}
 where $a_n$ and $b_m$ are non-zero. This expression of the polynomial $\xi$ implies that
 $$
 z_1^n\in \operatorname{span}\{z_1^i:0\leq i \leq n-1\} + \operatorname{span}\{z_2^j: 1\leq j \leq m\} + \operatorname{Ran} M_{z_1}M_{z_2}
 $$and
 $$
z_2^m\in \operatorname{span}\{z_1^i:0\leq i \leq n\} + \operatorname{span}\{z_2^j: 1\leq j \leq m-1\} + \operatorname{Ran} M_{z_1}M_{z_2}.
 $$These containments together with a straightforward application of mathematical induction imply that
$$H^2(\mu)=\operatorname{Ran}M_{z_1}M_{z_2}+ \operatorname{span}\{z_1^i , z_2^j: 0\leq{i}\leq{n}, 0\leq{j}\leq{m}\}.$$
Therefore the range of $(I-M_{z_1}M_{z_2}M_{z_2}^{*}M_{z_1}^*)$ is finite dimensional. Since the product $M_{z_1}M_{z_2}$ is a pure isometry on $H^2{(\mu)}$, we get a finite model triple $(\cF, P, U)$ by the Berger--Coburn--Lebow Theorem such that the pair $(M_{z_1},M_{z_2})$ is unitary equivalent to
$(M_{P^\perp U + zPU} , M_{U^*P + zU^*P^\perp})$. The rest of the proof now follows by noting that the set of bounded point evaluation for $H^2(\mu)$ is precisely $\cW$. Hence, a point $(z_1,z_2)\in{\mathbb{D}^2}$ is in $\cW$ if and only if $(\overline{z_1}, \overline{z_2})$ is a joint eigenvalue of $(M_{z_1}^{*}, M^{*}_{z_2})$.

By Lemma \ref{jemultipliers}, this is equivalent to   $(z_1,z_2)$ being a joint eigenvalue of $(P^\perp U+z PU,U^*P+zU^*P^\perp)$. Consequently,
\begin{align*}
\cW\cap\mathbb D^2&=\{(z_1,z_2)\in\mathbb{D}^2:(z_1,z_2)\in\sigma_T(P^\perp U+z_1z_2PU,U^*P+z_1z_2U^*P^\perp)\}\\
&=\cW_{P,U}\cap\mathbb D^2.
\end{align*}

The only thing that remains to be proved is that a distinguished variety $\cW$ equals $\cW_{P,U}$ if and only if both $\cW$ and $\cW_{P,U}$ are contained in $\mathbb D^2\cup\mathbb T^2\cup\mathbb E^2$. If $\cW = \cW_{P,U}$, then from $(i)\Rightarrow(iii)$ of the forward direction, we know the containment. Conversely, if both $\cW$ and $\cW_{P,U}$ are contained in $\mathbb D^2\cup\mathbb T^2\cup\mathbb E^2$, then repeating the arguments that we gave at the end of Section \ref{throughinner}, the proof is complete. \qed

\subsection{Examples and Remarks}
\begin{remark}
A symmetric algebraic variety in $\mathbb C^2$ need not be distinguished with respect to $\mathbb D^2$. For example, take $P=\sbm{1&0\\0&0}$ and $U=I_{\mathbb C^2}$. Then
  $$
  \cW_{P,U}=\{(z,1):z\in\mathbb C\}\cup\{(1,z):z\in\mathbb C\}
  $$
  is symmetric but clearly is not a distinguished variety with respect to $\mathbb D^2$. Indeed, the model triple $(\mathbb C^2,\sbm{1&0\\0&0},I_{\mathbb C^2})$ does not satisfy the compatibility conditions \eqref{StrangeConditions1}.
\end{remark}

\begin{remark}
What is special about the linear pencils $P^\perp U+zPU$ and $U^*P+zU^*P^\perp$? One could start with  any two matrix-valued rational inner functions  $\Phi$ and $\Psi$ on $\mathbb D$ such that
\begin{enumerate}
\item[(i)] the maps $z\mapsto \nu(\Phi(z))$ and $z\mapsto \nu(\Psi(z))$ are non-constant on $\mathbb D$;
\item[(ii)] for each $z\in\mathbb D$, the pair of matrices $(\Phi(z),\Psi(z))$ is commuting; and
\item[(iii)] $\Phi(z)\Psi(z)=z$ for all $z\in\mathbb D$.
\end{enumerate}
Then
$$
\cW_{\Phi,\Psi}:=\{(z_1,z_2)\in\mathbb C^2:(z_1,z_2)\in\sigma_T(\Phi(z_1z_2),\Psi(z_1z_2))\}
$$is a distinguished variety with respect to $\mathbb C^2$. The proof is along the same line as the proof of the forward direction of Theorem \ref{ourDV}.

 It is a consequence of the Berger--Coburn--Lebow Theorem that any such pair of functions is jointly unitarily equivalent to $(P^\perp U+zPU,U^*P+zU^*P^\perp)$ for some model triple $(\cF,P,U)$. We leave the details to the reader.
\end{remark}

\begin{example}\label{E:GenNeil}
A model triple $(\cF,P,U)$ for the Neil parabola $\{(z_1,z_2)\in\mathbb C^2:z_1^3=z_2^2\}$ is given by
$$
\cF=\mathbb C^5,\quad P=P_{\mathbb C^2\oplus \{0_{\mathbb C^3}\}}\quad \text{and}\quad U=E_\sigma,
$$where $E_\sigma$ is the permutation matrix induced by the permutation $\sigma=(13452)$ in $S_5$.
%
Indeed, a simple matrix computation gives us the following
\begin{align*}
 P^{\perp}U+z_1z_2PU&= \begin{bmatrix}
     0 & z_1z_2 & 0 & 0 & 0\\
     0 & 0 & 0 & 0 & z_1z_2\\
     1 & 0 & 0 & 0 & 0\\
     0 & 0 & 1 & 0 & 0\\
     0 & 0 & 0 & 1 & 0
    \end{bmatrix} \text{, } U^*P+z_1z_2U^*P^{\perp}=
    \begin{bmatrix}
     0 & 0 & z_1z_2 & 0 & 0\\
     1 & 0 & 0 & 0 & 0\\
     0 & 0 & 0 & z_1z_2 & 0\\
     0 & 0 & 0 & 0 & z_1z_2\\
     0 & 1 & 0 & 0 & 0
    \end{bmatrix}.
\end{align*}
A not very lengthy calculation yields that the set
\begin{align*}
\cW_{P,U}=\{(z_1, z_2)\in\mathbb{C}^2:(z_1, z_2)\in\sigma_T(P^\perp U+z_1z_2PU,U^*P+z_1z_2U^*P^\perp)\}
\end{align*}
is the same as the Neil parabola.

More generally, one can check by a somewhat tedious computation that a model triple for the distinguished variety
\begin{align}\label{GenNeil}
\mathcal{N}_{n,m}:=\{(z_1, z_2) \in \mathbb{C}^2: z_1^n=z_2^m \}; \quad n, m\geq{1}
\end{align}is given by
$$\cF= \mathbb{C}^{m+n}, \quad  P=P_{\mathbb C^m\oplus \{0_{\mathbb C^n}\},} \quad
U^*=\begin{bmatrix}
A & B  \\
C & D
\end{bmatrix},
$$
where $B$ is the $m\times n$ matrix with $1$ at the $(1,1)$ entry and zero elsewhere, $C$ is the $n\times m$ matrix with $1$ at the $(n,m)$ entry and zero
elsewhere, $D$ is the $n\times n$ upper triangular matrix with $1$ in the super diagonal entries and zero elsewhere, and $A$ is the $m\times m$ matrix given as
$$A=\begin{bmatrix}
\huge0 & 0 \\
I_{m-1} & 0
\end{bmatrix}.$$
\end{example}

\begin{remark}
Example \ref{E:GenNeil} shows that there is an $m+n$ dimensional model triple $(\cF,P,U)$ for the distinguished variety $\cN_{n,m}$ as in \eqref{GenNeil}. Why is the dimension $m+n$?  Theorem 2.1 of Knese in \cite{Knese-TAMS2010} asserts that if $\cW$ is a distinguished variety determined by a polynomial $\xi(z_1,z_2)$ of degree $(n,m)$ (i.e., the highest power of $z_1$ and $z_2$ in $\xi$ is $n$ and $m$, respectively) and if there is no polynomial of any smaller degree determining $\cW$, then there is a unitary
$$
 \begin{bmatrix} A&B\\C&D\end{bmatrix}:\begin{bmatrix}\mathbb C^m\\\mathbb C^n \end{bmatrix}\to \begin{bmatrix}\mathbb C^m\\\mathbb C^n \end{bmatrix}
$$
such that the transfer function arising out of the unitary $\sbm{ A & B \\ C & D}$ is the Agler--McCarthy inner function for $\cW$.
We shall see in section 4 that for every such choice of $\sbm{ A & B \\ C & D}$, we have a model triple $(\cF,P,U)$ with the space $\cF$ being the same as the space which the unitary $\sbm{A & B \\ C & D}$ acts on.
\end{remark}

\subsection{The numerical radius conditions}
At the end of this subsection, we shall prove that checking the compatibility condition \eqref{StrangeConditions1} is as simple as checking non-constancy of a couple of functions.

\begin{lemma}\label{L:BdrytoInt}
Let $\mathcal{F}$ be a Hilbert space and let $\varphi:\mathbb D \to\cB(\cF)$ be any analytic function.
\begin{enumerate}

\item If the map $z\mapsto \nu(\varphi(z))$ is non-constant in $\mathbb D$ and if $\nu(\varphi(z)) \le 1$ for all $z\in\mathbb D$, then $\nu(\varphi(z))<1$ for all $z\in\mathbb D$.

\item  If $\varphi$ can be continuously extended to $\mathbb T$ and if $\nu(\varphi(\zeta))\leq 1$ for all $\zeta\in\mathbb T$, then $\nu(\varphi(z))\leq 1$ for all $z\in\mathbb D$.
\end{enumerate}
\end{lemma}
\begin{proof}

It is well-known that the map $z \mapsto \nu(\varphi(z))$ is subharmonic for an analytic $\varphi$, see Theorem 4.4 and Corollary 4.5 in \cite{Ves}. Thus, the maximum principle applies and we get the conclusion (1).

For the second part, let us fix a unit vector $h$ in $\cF$ and define the analytic map $\varphi_h:\mathbb D\to\mathbb C$ as $\varphi_h(z)=\langle \varphi(z)h,h\rangle$ for all $z\in\overline{\mathbb D}$. Then by the hypothesis $|\varphi_h(\zeta)|\leq 1$ for $\zeta \in \mathbb T$. Therefore, by the maximum modulus principle,
\begin{align}\label{truth1}
|\varphi_h(z)|=|\langle\varphi(z)h,h\rangle|\leq1,\text{ for every $z\in\mathbb D$}.
\end{align} Since $h$ is arbitrary, $\nu(\varphi(z))=\sup\{|\langle\varphi(z)h,h\rangle|:\|h\|=1\}\leq1$ for every $z\in\mathbb D$.
\end{proof}

It is folklore that for a bounded operator $A$, we have $\nu(A) \le 1$ if and only if $Re(\beta A) \le I$ for all $\beta\in \mathbb T$; see Lemma 2.9 in \cite{BhPSR}. The next result is about when the inequalities can be made strict. This result will be used many times and is included for completeness since we could not find this in literature.
\begin{lemma}\label{StrIn}
	Let $A$ be a square matrix. If $Re(\beta A)< I$ for all $\beta\in \mathbb T$, then $\nu(A)<1$. Conversely, if $\nu(A)<1$, then $Re(\beta A)< I$ for all $\beta \in \mathbb T$.
\end{lemma}
\begin{proof}
Finite dimensionality is of crucial importance here. Suppose $Re(\beta A)< I$ for all $\beta \in \mathbb T$ and $\nu(A)=1$. Since $A$ is a matrix, the numerical range of $A$ is a compact set. Hence there is a vector $h_0$ with $\|h_0\|=1$ such that $|\langle Ah_0,h_0\rangle|=1$. Choosing $\beta$ suitably on the unit circle, we get $\beta\langle Ah_0,h_0\rangle=1$ which is a contradiction. Therefore $\nu(A)<1$.

Conversely, let $\nu(A)<1$. Take any vector $h$ with $\|h\|=1$. Then for all $\beta \in \mathbb T$,
	$$\langle Re(\beta A)h,h\rangle=Re\beta\langle Ah,h\rangle .$$
	Hence we have,
	$$\big\vert\langle Re(\beta A)h,h\rangle\big\vert\leq|\beta|\nu(A)<1.$$
	Therefore, $Re(\beta A)< I$ for all $\beta \in \mathbb T$.
\end{proof}

\begin{lemma}\label{NUM RAD}
Let $(\cF,P,U)$ be any model triple.
\begin{enumerate}
\item Then for every $z\in\overline{\mathbb D}$,
$$\nu(P^\perp U+zU^*P)\leq 1\text{ and } \nu( U^*P+zP^\perp U)\leq 1;$$
\item If the map $\varphi_1:\mathbb D\to\mathbb R$ defined as $\varphi_1(z)=\nu(P^\perp U+z U^*P)$ is nonconstant, then
$$\nu(P^\perp U+z U^*P)<1 \text{ for every $z\in\mathbb D$.}$$ Moreover, if in addition, $\dim\cF<\infty$ then
$$\nu(P^\perp U+zP U)<1\text{ and }\nu(U^*P^\perp+zU^*P)<1 \text{ for all $z$ in $\mathbb D$};$$.
\item If the map $\varphi_2:\mathbb D\to\mathbb R$ defined as $\varphi_2(z)=\nu(U^*P+z P^\perp U)$ is nonconstant, then
$$\nu(U^*P+z P^\perp U)<1 \text{ for every $z\in\mathbb D$.}$$ Furthermore, if $\dim\cF<\infty$ then
$$\nu(U^*P+z U^*P^\perp)<1\text{ and }\nu(PU+zP^\perp U)<1\text{ for all $z$ in $\mathbb D$.}$$
\end{enumerate}
\end{lemma}
\begin{proof}
For $(1)$, we first prove it for boundary points and then use Lemma \ref{L:BdrytoInt} to conclude for interior points. To this end, we pick a uni-modular $\eta$ and compute
$$
\text{Re}\eta(P^\perp U+U^*P)=\frac{1}{2}((\eta P^\perp+\bar\eta P)U+U^*(\eta P+\bar\eta P^\perp))=\operatorname{Re }(\eta P^\perp+\bar\eta P)U.
$$We note that $(\eta P^\perp+\bar\eta P)U$ is a product of two unitaries and therefore we conclude that {\em for every projection $P$ and every unitary $U$ acting on a Hilbert space
\begin{align}\label{lagbe}
\operatorname{Re}\eta(P^\perp U+U^*P) \leq I.
\end{align}}Now for every $\zeta,\eta\in\mathbb T$, we note that
$$
\operatorname{Re}\eta(P^\perp U+\zeta U^*P)=\operatorname{Re}\eta\zeta^{\frac{1}{2}}(\bar\zeta^{\frac{1}{2}}P^\perp U+\zeta^{\frac{1}{2}}U^*P)=\text{Re}\eta'(P^\perp U'+U'^*P)
$$where $\eta'=\eta\zeta^{\frac{1}{2}}$ and $U'=\bar\zeta^{\frac{1}{2}}U$. Therefore applying \eqref{lagbe} we conclude that for every $\zeta,\eta\in\mathbb T$
$$
\operatorname{Re}\eta(P^\perp U+\zeta U^*P)\leq I.
$$This proves that for every $\zeta\in\mathbb T$, $\nu(P^\perp U+\zeta U^*P)\leq 1$. Now apply part (2) of Lemma \ref{L:BdrytoInt} to the analytic function $z\mapsto P^\perp U+zU^*P$.

For the first part of part (2), just apply part (1) of Lemma \ref{L:BdrytoInt} to the linear pencil $z\mapsto P^\perp U+zU^*P$. For the second part, we use the finite dimensionality assumption. By Lemma \ref{StrIn}, $\nu(P^\perp U+zU^*P)<1$ implies $Re(\beta(P^\perp U+zU^*P))<I$ for any $\beta\in{\mathbb{T}}$ and $z\in\mathbb D$. This means that for every fixed $\beta\in\mathbb T$,
\begin{align*}
\beta(P^\perp U+zU^*P)+\overline{\beta}(U^*P^\perp+\overline{z} PU)<2I  \text{ for every $z\in\mathbb D$},
\end{align*}
which is same as saying that
\begin{align*}
\beta(P^\perp U+\overline{\beta}^2 \bar{z} PU)+\overline{\beta}(U^*P^\perp+{\beta^2 z}U^*P)<2I \text{ for every $z\in\mathbb D$}.
\end{align*}
The above equation is true for all $\beta\in{\mathbb{T}}$. Hence from Lemma \ref{StrIn} again, we have $\nu(P^\perp U+zPU)<1$ for all $z\in{\mathbb{D}}$. For a proof of the second inequality, one just uses the fact that for every Hilbert space bounded operator $A$, $\nu(A)=\nu(A^*)$ and does a similar computation as above to conclude that $\nu(U^*P^\perp+zU^*P)<1$ for every $z\in\mathbb D$.

Proof of part (3) is similar to that of part (2).
\end{proof}

We note down a direct consequence of Lemma \ref{NUM RAD} that gives an easily checkable condition to verify the compatibility condition \eqref{StrangeConditions1}.
\begin{corollary} \label{crucial}
Let $(\cF,P,U)$ be a finite model triple such that both the functions
$$
 z\mapsto \nu(P^\perp U+z U^*P)\quad \text{ and }\quad  z\mapsto \nu(U^*P+z P^\perp U)
$$are nonconstant on $\mathbb D$. Then
\begin{equation} \label{StrangeConditions} \nu(U^*(P + z P^\perp)) < 1  \text{ and } \nu((P^\perp + zP)U) < 1 \end{equation}
for all $z$ in the open unit disc $\mathbb D$. \end{corollary}

\begin{proof}
Parts (2) and (3) of Lemma \ref{NUM RAD} along with the finite dimensionality assumption justify the corollary.
\end{proof}

It should be observed that a projection $P$ that satisfies \eqref{StrangeConditions} is necessarily non-trivial.

\section{A new realization formula and the passage between the two descriptions} \label{AMandBCL}
\subsection{The Realization Formula}
A pair $(\Psi, \cE)$ is called a {\em contractive analytic} function if $\cE$ is a Hilbert space and
$\Psi:\mathbb D\to \cB(\cE)$ is analytic and contractive, i.e., $\|\Psi\|_\infty:=\sup_{\mathbb D}\|\Psi(z)\|\leq 1$.  It is a folklore that $(\Psi, \cE)$ is a contractive analytic function if and only if there is an auxiliary Hilbert space $\cH$ and a unitary operator
 $$
 U=\begin{bmatrix} A&B\\C&D\end{bmatrix}:\begin{bmatrix}\cE\\\cH\end{bmatrix}\to\begin{bmatrix}\cE\\\cH\end{bmatrix}
 $$such that
 $$
 \Psi(z)=A+zB(I_{\cH}-zD)^{-1}C.
 $$The operator $U$ is called the {\em unitary colligation} for $\Psi$ and the function $\Psi$ is called the {\em transfer function} for the unitary $U$. The new realization formula that we obtain is the following.
\begin{thm}[A new realization formula] \label{Thm:RealForm}
Every projection $P$ and a unitary $U$ acting on a Hilbert space $\cF$ gives rise to a contractive analytic function
$\Psi_{P,U}:\mathbb D\to\cB(\operatorname{Ran}P)$ defined by
\begin{align}\label{RealForm}
\Psi_{P,U}(z):=P (I_\cF-zU^*P^\perp )^{-1} U^*P|_{\operatorname{Ran}P}.
\end{align}
Conversely, if $\cE$ is a Hilbert space and $\Psi:\mathbb D\to\cB(\cE)$ is a contractive analytic function, then $\Psi=\Psi_{P,U}$ for some unitary $U$ acting on a Hilbert space and $P$ is the orthogonal projection of $\cF$ onto $\cE$.

Moreover, when $\mathcal E$ is finite dimensional and $\Psi$ is rational inner, the Hilbert space $\cF$ can be chosen to be finite dimensional.
\end{thm}

\begin{proof}

We first note that for a model triple $(\cF,P,U)$, the analytic function
\begin{align}\label{OldReal}
\Psi_{P,U}(z)=P (I_\cF-zU^*P^\perp )^{-1} U^*P|_{\operatorname{Ran}P}
\end{align}is contractive because a straightforward computation yields that $I -  \Psi_{P,U}(z) \Psi_{P,U}(z)^*$ is a positive operator and hence $\Psi_{P,U}$ is a contractive analytic function.

To show that every contractive analytic function $\Psi:\mathbb D\to\cB(\cE)$ is of the form \eqref{OldReal}, we invoke the classical realization formula to obtain an auxiliary Hilbert space $\mathcal{H}$ and a unitary operator
$$\begin{bmatrix}
A&B\\C&D
\end{bmatrix}:\begin{bmatrix}\mathcal{E}\\{\mathcal{H}}\end{bmatrix}\to \begin{bmatrix}\mathcal{E}\\{\mathcal{H}}\end{bmatrix}$$
such that
\begin{align}\Psi(z)=A+zB(I-zD)^{-1}C \label{relformula}. \end{align}
By setting
$U^*=\begin{bmatrix}
A&B\\C&D
\end{bmatrix}$
and $P$ as the projection from $\mathcal{E}\oplus{\mathcal{H}}$ to $\mathcal{E}$, we get that
\begin{align}\label{RealUni}
\begin{bmatrix}
A&B\\C&D
\end{bmatrix}=\begin{bmatrix}
P U^*P|_{\operatorname{Ran}P}& P U^*P^\perp|_{\operatorname{Ran}P^\perp}\\
P^\perp U^*P|_{\operatorname{Ran}P}& P^\perp U^*P^\perp|_{\operatorname{Ran}P^\perp}
\end{bmatrix}.
\end{align}
With this new realization of the unitary colligation, we note from \eqref{relformula} that
\begin{align}
\Psi(z)=(PU^*P+zPU^*P^\perp(I_{\cL}-zP^\perp U^*P^\perp)^{-1}P^\perp U^*P)|_{\operatorname{Ran}P},
\end{align}
which, after a simplification, turns out to be the same as the formula stated in \eqref{OldReal}.

If $\cE$ is finite dimensional and $\Psi$ is rational inner, then it is well known (see for example \cite[Section 11]{BKJFA2013}) that the auxiliary space $\cH$ above can also be chosen to be finite dimensional. Since $\cF = \cE \oplus \cH$, the proof of Theorem \ref{Thm:RealForm} is complete. \end{proof}

\subsection{The passage between the two descriptions}
With the help of this new realization formula, we show here a passage between the inner function description and the linear pencils description of a distinguished variety.

\begin{thm}\label{Sec-Thm:PUtoPsi}
Let $\mathcal W$ be a distinguished variety with respect to $\mathbb D^2$.
 \begin{itemize}
 \item[(1)]If $(\cF, P, U)$ is a finite model triple that satisfies $\mathcal W\cap \mathbb D^2 = \cW_{P,U}\cap\mathbb D^2$, then $\mathcal W\cap\mathbb D^2=\cW_{\Psi_{P,U}}\cap\mathbb D^2.$
 \item[(2)] Let $\Psi$ be a rational matrix-valued inner function such that $\mathcal W\cap\mathbb D^2=\cW_{\Psi}\cap\mathbb D^2$. If $(\cF, P, U)$  a finite model triple such that $\Psi=\Psi_{P,U}$, then $\mathcal W\cap\mathbb D^2 = \cW_{P,U}\cap\mathbb D^2$.
  \end{itemize}
\end{thm}
\begin{proof}
For (1), we first prove that, for a finite model triple $(\cF,P,U)$, we have the inclusion $\cW_{P,U}\cap \mathbb D^2\subseteq\cW_{\Psi_{P,U}}\cap \mathbb D^2$. Let $(z_1,z_2) \in \cW_{P,U}\cap \mathbb D^2$. If $z_2 = 0$, there is a non-zero vector $w$ such that $P^\perp Uw=z_1w$ and $U^*Pw=0$. Thus $Pw=0$ and hence $\Psi_{P,U}(z_1)w=0$ proving that $(z_1,0)$ is in $\cW_{\Psi_{P,U}}\cap \mathbb D^2$.

Let $(z_1,z_2)\in\cW_{P,U}\cap \mathbb D^2$ and $z_2\neq 0$. This means that there exists a non-zero vector $w$ in $\mathbb C^n$ such that
$$
(P^\perp +z_1z_2P)Uw=z_1w \text{ and }U^*(P+z_1z_2P^\perp)w=z_2w.
$$
Re-arranging the second equation, we get $z_2(I_\cF -z_1U^*P^\perp )w=U^*Pw$.
Since $z_1\in\mathbb D$ and $U^*P^\perp$ is a contraction, the matrix $(I_\cF -z_1U^*P^\perp )$ is invertible and hence
\begin{align}\label{CozInvrtbl}
z_2w=(I_\cF -z_1U^*P^\perp )^{-1} U^*Pw.
 \end{align}
Hence $Pw$ must be non-zero (otherwise \eqref{CozInvrtbl} implies that $z_2w=0$, which contradicts the fact that neither $z_2$ nor $w$ is zero). Therefore we have
 $$z_2Pw=P(I_\cF -z_1U^*P^\perp )^{-1} U^*Pw.$$
 Consequently $(\Psi_{P,U}(z_1)-z_2I_\cF )Pw=0$ or equivalently $\det(\Psi_{P,U}(z_1)-z_2I_\cF )=0$.

We now prove the other inclusion, i.e., $\cW_{\Psi_{P,U}}\cap \mathbb D^2\subset\cW_{P,U}\cap \mathbb D^2$. If $(z_1,0)$ is in $\cW_{\Psi_{P,U}}\cap \mathbb D^2$, then by definition $\det ( \Psi_{P,U}(z_1)) = 0$. So, there is a non-zero vector $w$ in the range of the projection $P$ (because that is the space on which $\Psi_{P,U}(z_1)$ acts) such that $0 = \Psi_{P,U}(z_1)w=P(I-z_1U^*P^\perp)^{-1}U^*Pw$. This means that $(I-z_1U^*P^\perp)^{-1}U^*Pw$ is in the range of $P^\perp$. Define the vector
$$ v:=P^\perp(I-z_1U^*P^\perp)^{-1}U^*Pw=(I-z_1U^*P^\perp)^{-1}U^*Pw.$$

This $v$ is a non-zero vector because otherwise $Pw$ would be $0$ contradicting that $w$ is a non-zero vector from the range of $P$. From the definition of $v$, we have $U(I-z_1U^*P^\perp)v=w$, which after multiplying by $P^\perp$ from left gives $P^\perp Uv=z_1v$. Clearly $U^*Pv=0$ because $v\in\operatorname{Ran}P^\perp$. Consequently, $v$ is in the kernel of both $(P^\perp U-z_1I)$ and $U^*P$, and hence $(z_1,0)$ is in $\cW_{P,U}\cap\mathbb D^2$.

Let us now suppose that $(z_1,z_2)\in\cW_{\Psi_{P,U}}\cap\mathbb D^2$ where $z_2\neq0$. Let $w$ be a non-zero vector such that $P(I-z_1U^*P^\perp)^{-1}U^*Pw=z_2w$. Let $v'=\frac{1}{z_2}P^\perp(I-z_1U^*P^\perp)^{-1}U^*Pw$ and define $v=w+v'$. Then $v$ is a non-zero vector such that
\begin{align*}
(I_{\cF}-z_1U^*P^\perp)^{-1}U^*Pv&=(I_{\cF}-z_1U^*P^\perp)^{-1}U^*Pw\\
&=P(I_{\cF}-z_1U^*P^\perp)^{-1}U^*Pw+P^\perp(I_{\cF}-z_1U^*P^\perp)^{-1}U^*Pw\\
&=z_2(w+v')=z_2v,
\end{align*}which implies that $U^*Pv=(I_{\cF}-z_1U^*P^\perp)z_2v$. We simplify this equation to obtain
\begin{align}\label{eqN1}
U^*(P+z_1z_2P^\perp)v=z_2v.
\end{align}We next multiply the above equation from left by $(P^\perp+z_1z_2P)U$ to obtain
\begin{align}
z_1z_2v=z_2(P^\perp+z_1z_2P)Uv.
\end{align}Now we use the fact that $z_2\neq0$ to arrive at
$$
(P^\perp+z_1z_2P)Uv=z_1v.
$$
This and \eqref{eqN1} together prove that $\cW_{P,U}\cap\mathbb D^2=\cW_{\Psi_{P,U}}\cap\mathbb D^2$.

The proof of (2) depends on noting that, given $\Psi$ and $(\cF,P,U)$ as in (2) $\Psi = \Psi_{P,U}$. Hence, by what is already proved above, $\cW_\Psi\cap\mathbb D^2=\cW_{P,U}\cap\mathbb D^2.$
\end{proof}
We note that for an arbitrary model triple $(\cF,P,U)$, it need not be true that $\cW_{P,U}=\cW_{\Psi_{P,U}}$. Take for example the model triple $(\cF,0,U)$, for which $\cW_{0,U}=\sigma(U)\times\mathbb C$, whereas since $\Psi_{0,U}\equiv0$, $\cW_{\Psi_{0,U}}=\mathbb C\times\{0\}$. Another example for $\cW_{P,U}\neq\cW_{\Psi_{P,U}}$ is $(\mathbb C^2,\sbm{1&0\\0&0}, I_{\mathbb C^2})$, as can be checked easily. However, the following is true for a model triple of arbitrary dimension.
\begin{thm}Let $(\cF,P,U)$ be any model triple (finite or infinite) such that $P\neq 0$, and let $\Psi_{P,U}:\mathbb D\to \cB(\operatorname{Ran}P)$ be the associated contractive analytic function. Consider the sets
$$
\cW_{P,U}=\{(z_1,z_2)\in\mathbb C^2:(z_1,z_2)\in\sigma_p(P^\perp U+zPU,U^*P+zU^*P^\perp)\},
$$where for a commuting pair $(A,B)$ of bounded operators, $\sigma_p(A,B)$ denotes the set of joint eigenvalues of $(A,B)$; and
   $$
\cW_\Psi=\{(z_1,z_2)\in\mathbb C^2: \operatorname{Ker}(\Psi(z_1)-z_2I_{\cE})\neq\{0\}\}
   $$where $\Psi:\mathbb D\to\cB(\cE)$ is contractive analytic function. Then
\begin{align}\label{maybeEmpty}
\cW_{P,U}\cap\mathbb D^2=\cW_{\Psi_{P,U}}\cap\mathbb D^2.
\end{align}Conversely, if $\Psi:\mathbb D\to\cB(\cE)$ is any contractive analytic function and $(\cF,P,U)$ is its realizing model triple, then
$$
\cW_\Psi\cap\mathbb D^2=\cW_{P,U}\cap\mathbb D^2.
$$
\end{thm}
\begin{proof}
The proof is along the same line as that of Theorem \ref{Sec-Thm:PUtoPsi}.
\end{proof}
We end this section by noting that both the sides of \eqref{maybeEmpty} may very well be empty. For example if we choose
   $$
\cF=\begin{bmatrix}\mathbb C\\\mathbb C\end{bmatrix},\quad P=\begin{bmatrix} 1&0\\0&0 \end{bmatrix} \quad\text{and}\quad U= \begin{bmatrix} 1&0\\0&1 \end{bmatrix},
$$then
$$
(P^\perp U+zPU,U^*P+zU^*P^\perp)=\left(\begin{bmatrix}z&0\\0&1 \end{bmatrix},\begin{bmatrix}1&0\\0&z \end{bmatrix}\right).
$$Therefore, the joint eigenvalues of $(P^\perp U+zPU,U^*P+zU^*P^\perp)$ does not intersect the open bidisc. On the other hand, the corresponding ration inner function is
$$
\Psi_{P,U}(z)=P(I_{\mathbb C^2}-zU^*P^\perp)^{-1}U^*P|_{\operatorname{Ran}P}=1 \text{ for all } z\in\mathbb D.
$$
And hence $\operatorname{Ker}(\Psi_{P,U}(z_1)-z_2I_{\mathbb C^2})$ does not intersect the open bidisc either.

\section{Model triples and contractive analytic functions} \label{Category}
Two contractive analytic functions $(\Psi_1, \cE_1)$ and $(\Psi_2, \cE_2)$ are said to be {\em  unitarily equivalent} if there is a unitary operator $\tau: \cE_1 \rightarrow \cE_2$ such that $\tau \Psi_1(z) = \Psi_2(z) \tau$ for all $z \in \mathbb D$.

 It is easy to see that unitarily equivalent model triples give rise to unitarily equivalent contractive analytic functions. How about a converse? We show that given a contractive analytic function, a canonical choice of a model triple can be naturally made so that when two contractive analytic functions are unitarily equivalent, so are the associated canonical model triples, see Proposition \ref{Prop: coincidence}.

Consider the two categories
$$
\mathfrak{B}:=\{(\cF,P,U): \text{ $P$ is a projection and $U$ is a unitary on $\cF$}\}
$$with the morphisms between two elements $(\cF_1,P_1,U_1)$ and $(\cF_2,P_2,U_2)$ defined as a linear operator $\tau:\cF_1\to\cF_2$ that satisfies
\begin{align}\label{MorphismBCL}
\tau(P_1,U_1)=(P_2,U_2)\tau;
\end{align}
and
$$
\mathfrak C=\{(\Psi,\cE):\Psi:\mathbb D\to\cB(\cE) \text{ is analytic and contractive}\}
$$with the morphisms between two elements $(\Psi_1,\cE_1)$ and $(\Psi_2,\cE_2)$ defined as a linear operator $\tau:\cE_1\to\cE_2$ that satisfies
 \begin{align}\label{MorphismCAF}
 \tau\Psi_1(z)=\Psi_2(z)\tau \text{ for all }z\in\mathbb D.
 \end{align}

  Corresponding to an object $\chi=(\cF,P,U)$  in $\mathfrak B$, we have an object $\Psi_{P,U}$ of $
\mathfrak C$  given by Theorem \ref{Thm:RealForm}, i.e., $\Psi_{P,U}:\mathbb D\to\cB(\operatorname{Ran}P)$ is the function
 $$
 \Psi_{P,U}(z)=P(I_\cF-zU^*P^\perp)^{-1}U^*P|_{\operatorname{Ran}P}.
 $$
 Let $\chi_1=(\cF_1,P_1,U_1)$ and $\chi_2=(\cF_2,P_2,U_2)$ be two objects in $\mathfrak B$ and let $\tau$ be a morphism between them. It is easy to see from \eqref{MorphismBCL} that $\tau$ takes the following operator matrix form
 $$
 \tau=\begin{bmatrix} \tau_* & 0\\0& \tau_{**} \end{bmatrix}:
 \begin{bmatrix} \operatorname{Ran}P_1\\\operatorname{Ran}P_1^\perp\end{bmatrix}
 \to
 \begin{bmatrix} \operatorname{Ran}P_2\\\operatorname{Ran}P_2^\perp\end{bmatrix}.
 $$
 The linear transformation $\tau_*:\operatorname{Ran}P_1\to\operatorname{Ran}P_2$ induced by $\tau$ is easily seen to have the property
 $$
 \tau_*\Psi_{P_1,U_1}(z)=\Psi_{P_2,U_2}(z)\tau_* \text{ for all }z\in\mathbb D.
 $$
 Thus $\tau_*$ is a morphism between the objects $(\Psi_{P_1,U_1},\operatorname{Ran}P_1)$ and $(\Psi_{P_2,U_2},\operatorname{Ran}P_2)$. These morphisms will be referred to as the {\em induced} morphisms.
 \begin{proposition}
 The map $\mathfrak f:\mathfrak B\to\mathfrak C$ defined as
 $$
\mathfrak f:( (\cF,P,U),\tau)\mapsto(\Psi_{P,U},\tau_*)
 $$has the functorial properties, i.e.,
 \begin{enumerate}
 \item if $\iota:(\cF,P,U)\to(\cF,P,U)$ is the identity morphism, then the induced morphism $\iota_*:(\Psi_{P,U},\operatorname{Ran}P)\to(\Psi_{P,U},\operatorname{Ran}P)$ is the identity morphism; and
 \item if $\tau:\chi_1\to\chi_2$ and $\tau':\chi_2\to\chi_3$ are two morphisms in $\mathfrak B$, then
 $$
 (\tau'\circ\tau)_*=\tau'_*\circ\tau_*.
 $$
 \end{enumerate}
 Moreover, if $\chi_1$ and $\chi_2$ are unitarily equivalent via a unitary similarity $\tau$, then so are $\Psi_{P_1,U_1}$ and $\Psi_{P_2,U_2}$ via the induced unitary $\tau_*$.
 \end{proposition}
 \begin{proof}
 The proof is straightforward and hence we omit it.
\end{proof}
It is natural to expect a converse of the `moreover' part in the above result, especially because by Theorem \ref{Thm:RealForm} there corresponds a model triple for every contractive analytic function. However, unlike the forward direction, this model triple is not uniquely determined by the contractive analytic function. For example, one can check that both the unitaries
\begin{align*}
\begin{bmatrix}
\begin{array}{c|cc}
0 & 0 & 1\\
\hline
1 & 0 & 0 \\
0 & 1 & 0
\end{array}
\end{bmatrix},
\begin{bmatrix}
\begin{array}{c|cc}
0 & 1 & 0\\
\hline
0 & 0 & 1 \\
1 & 0 & 0
\end{array}
\end{bmatrix}:\begin{bmatrix}\mathbb C\\\mathbb C^2\end{bmatrix}\to \begin{bmatrix}\mathbb C\\\mathbb C^2\end{bmatrix}
\end{align*}serve as a unitary colligation for the contractive function $z\mapsto z^2$. Consequently, the function $z\mapsto z^2$ has two distinct model triples. There is, nevertheless, a canonical choice of a model triple for a contractive analytic function.

For an object $(\Psi,\cE)$ in $\mathfrak C$, consider the associated de Branges--Rovnyak reproducing kernel
\begin{align}\label{deBrRovKer}
K^\Psi(z,w)=\frac{I_{\cE}-\Psi(z)\Psi(w)^*}{1-z\bar w}.
\end{align}
Let $\cH$ be a Hilbert space and $g:\mathbb D\to\cB(\cH,\cE)$ be a function such that
\begin{align}\label{Kolmo}
K^\Psi (z,w)=g(z)g(w)^*.
\end{align}This is so called a Kolmogorov decomposition of the kernel $K^\Psi$. For a quick and easy reference, the function $g:\mathbb D\to\cB(\cH_\Psi,\cE)$ satisfying the Kolmogorov decomposition \eqref{Kolmo} will be called a {\em Kolmogorov function} for $\Psi$. Now we execute a classical {\em lurking isometry} program.
Using the definition \eqref{deBrRovKer} of $K^\Psi$ and after a rearrangement of the terms one arrives at
$$
\langle e,f\rangle_{\cE} + \langle \bar w g(w)^*e, \bar z g(z)^* f\rangle_\cH = \langle \Psi(w)^*e,\Psi(z)^*f\rangle_\cE+\langle g(w)^* e, g(z)^* f \rangle_\cH
$$for every $z,w\in\mathbb D$ and $e,f\in\cE$. This readily implies that the map
\begin{align*}
\mathfrak u:\overline{\operatorname{span}}\left\{\begin{bmatrix}I_\cE\\ \bar z g(z)^* \end{bmatrix}f:z\in\mathbb D \text{ and }f\in\cE\right\}\to
\overline{\operatorname{span}}\left\{\begin{bmatrix}\Psi(z)^*\\  g(z)^* \end{bmatrix}f:z\in\mathbb D \text{ and }f\in\cE\right\}
\end{align*}defined densely by
\begin{align}\label{PartialUni}
\mathfrak u:\sum_{j=1}^N\begin{bmatrix}I_\cE\\  \bar z_jg(z_j)^* \end{bmatrix}f_j\mapsto\sum_{j=1}^N\begin{bmatrix}\Psi(z_j)^*\\  g(z_j)^* \end{bmatrix}f_j
\end{align}is a unitary. We wish to extend this partially defined unitary to whole of $\cE\oplus\cH$, which we can do if the orthocomplements of the domain and
codomain of $\mathfrak u$ in $\cE\oplus\cH$ have the same dimension; if not, we can add an infinite dimensional Hilbert space say, $\cR$ to $\cH$ so that
$\mathfrak u$ has a unitary extension to $\cE\oplus\cH\oplus\cR$. We pause here to note that there is a {\em minimal} choice of the auxiliary Hilbert space
$\cH$, viz.,
$$
\cH_\Psi:=\overline{\operatorname{span}}\{g(z)^*e:z\in\mathbb D \text{ and }e\in \cE\},
$$and that this is actually isomorphic to the defect space of $M_{\Psi}^*$. Indeed, from the Kolmogorov decomposition \eqref{Kolmo} of $K^\Psi$, we see that
\begin{align*}
\langle (I_{H^2(\cE)}-M_\Psi M_\Psi^*)\mathbb S_w e,\mathbb S_z f\rangle_{H^2(\cE)}=\langle g(w)^*e,g(z)^*f \rangle_{\cH_\Psi},
\end{align*}where $\mathbb S$ is the Szeg\"o kernel for $\mathbb D$. This in particular implies that the map densely defined as
$$
(I_{H^2(\cE)}-M_\Psi M_\Psi^*)^\frac{1}{2}\sum_{j=1}^N\mathbb S_{w_j}e_j \mapsto \sum_{j=1}^Ng(w_j)^*e_j
$$is a unitary.

For a contractive analytic function $(\Psi,\cE)$, we denote by $\cF_\dag$ the minimal space containing $\cE\oplus \cH_\Psi$ to which the partially defined unitary $\mathfrak u$ as in \eqref{PartialUni} can be
extended. Let $U_\dag$ be a unitary operator on $\cF_\dag$ that extends $\mathfrak u$ and $P_\dag$ be the orthogonal projection of $\cF_\dag$ onto $\cE$.

\begin{definition}
 A model triple $(\cF_\dag,P_\dag,U_\dag)$ obtained from a contractive analytic function $(\Psi,\cE)$ as above will be referred to as a
 {\em canonical model triple} for $(\Psi,\cE)$.
\end{definition}

Before we move on, let us return to the example of the function $z\mapsto z^2$. If we do the computations above for this function, then it is straightforward that the canonical model triple is as follows. The space $\cF_\dag$ is $\mathbb C^3 (=\mathbb C \oplus \mathbb C^2)$, the projection is the one dimensional projection from $\mathbb C^3$ onto the first component and the unitary is the $3 \times 3$ unitary matrix corresponding to the permutation $(231)$.

\begin{remark}
While the space $\cF_\dag$ and the projection $P_\dag$ in a canonical model triple for a contractive analytic function $(\Psi,\cE)$ are
uniquely determined by $\Psi$, the unitary operator $U_\dag$ is not uniquely determined because a priori, there can be many unitary extensions of the partially
defined unitary $\mathfrak u$ as in \eqref{PartialUni} and all such unitary extensions qualify to be a member of the canonical model triple.
\end{remark}

\begin{remark}

There is a well-known family of contractive analytic functions $(\Psi,\cE)$ for which a minimal choice of the space $\cF_\dag$ is $\cE\oplus\cH_\Psi$
itself, viz., the matrix-valued rational inner functions. Indeed, for a matrix-valued rational inner function $(\Psi,\mathbb C^n)$ the space $\cH_\Psi$ is isomorphic to the {\em model space}
$H^2(\mathbb C^n)\ominus \Psi\cdot H^2(\mathbb C^n)$, which is known to be finite dimensional; see for example \cite[Section 11]{BKJFA2013}. Therefore, for a matrix-valued rational inner function, the partially defined unitary $\mathfrak u$ as in \eqref{PartialUni} acts on a subspace of the finite dimensional
Hilbert space $\mathbb C^n\oplus (H^2\ominus \Psi\cdot H^2(\mathbb C^n))$, and therefore has a unitary extension to the space. Consequently, {\em a matrix-valued rational inner
function has a finite canonical model triple.}
\end{remark}
\begin{proposition} \label{Prop: coincidence}
If two contractive analytic functions $(\Psi_1,\cE_1)$ and $(\Psi_2,\cE_2)$ are unitarily equivalent, then their canonical model triples are also unitarily equivalent.
\end{proposition}
\begin{proof}
Let $\tau:\mathcal{E}_1\to\mathcal{E}_2$ be a unitary such that
\begin{align}\label{Coin_CAF}
\tau\Psi_1(z)=\Psi_2(z)\tau \text { for all } z\in{\mathbb{D}}.
\end{align}Note that if $K^1$ and $K^2$ are the de-Branges--Rovnyak reproducing
kernels associated to $(\Psi_1,\cE_1)$ and $(\Psi_2,\cE_2)$, then we have
$$\tau K^1(z, w)= K^2(z, w)\tau\text{ for all $z,w\in{\mathbb D}$}.$$
For each $j=1,2$, let $g_j:\mathbb D\to \cB(\cH_{\Psi_j},\cE_j)$ be a Kolmogorov function for $\Psi_j$. There is a unitary operator
$\hat\tau:\cH_{\Psi_1}\to\cH_{\Psi_2}$ induced by $\tau$ as
\begin{align*}
\hat{\tau}: g_1(z)^*e \longmapsto  g_2(z)^*\tau e \text{ for all } z\in{\mathbb D} \text{ and }e\in\cE_1 .
\end{align*}
Let $\mathfrak u_j$ denote the partially defined unitaries as in \eqref{PartialUni} corresponding to $\Psi_j$. If $(\cF_{\dag1},P_{\dag1},\mathfrak U_{\dag1})$ is a canonical model triple for $\Psi_1$, then define
$$
(\cF_{\dag2},P_{\dag2},\mathfrak U_{\dag2}):=(\tilde\tau\cF_{\dag1},\tilde\tau P_{\dag1}\tilde\tau,\tilde\tau \mathfrak U_{\dag1}\tilde\tau)
$$ where if $\cF_{\dag1}=\cE_1\oplus\cH_{\Psi_1}$, then $\tilde{\tau}=\sbm{\tau & 0\\ 0 & \hat{\tau}}$; if $\cF_{\dag1}=\cE_1\oplus\cH_{\Psi_1}\oplus\cR$, then $\tilde\tau=\sbm{\tau & 0\\ 0 & \hat{\tau}}\oplus I_\cR$. To show that $(\cF_{\dag2},P_{\dag2},\mathfrak U_{\dag2})$ is a canonical model triple for $\Psi_2$,
all we have to show is that the unitary $\mathfrak U_{\dag2}$ extends $\mathfrak u_2$, which is established in the following computation:
\begin{align*}
\mathfrak U_{\dag2}\begin{bmatrix}I_{\cE_2}\\  \bar zg_2(z)^*\end{bmatrix}e=
\tilde\tau\mathfrak U_{\dag1}\tilde\tau^*\begin{bmatrix}I_{\cE_2}\\ \bar zg_2(z)^*\end{bmatrix}e=
\tilde\tau\mathfrak{U_1}\begin{bmatrix}I_{\cE_1}\\  \bar zg_1(z)^*\end{bmatrix}\tau^*e=
\tilde\tau\begin{bmatrix}\Psi_1(z)^*\\  g_1(z)^*\end{bmatrix}\tau^*e&=
\begin{bmatrix}\Psi_2(z)^*\\  g_2(z)^*\end{bmatrix}e.
\end{align*}
\end{proof}

With this, we leave the distinguished boundaries with respect to the bidisc and move on to the symmetrized bidisc.

\section{Distinguished varieties in the symmetrized bidisc}\label{S:SymmBDisc}

\subsection{The Pal and Shalit description}
The {\em open symmetrized bidisc} and its closure are
$$ \mathbb G = \{ (z_1 + z_2, z_1z_2) : |z_1| < 1 \mbox{ and } |z_2| < 1\}$$
and
 $$\Gamma = \{ (z_1 + z_2, z_1z_2) : |z_1| \le 1 \mbox{ and } |z_2| \le 1\}$$
respectively. The {\em distinguished boundary} $b\Gamma$ is
$$b\Gamma = \{ (z_1 + z_2, z_1 z_2) : |z_1| = |z_2| = 1\}.$$

The new characterization of distinguished varieties with respect to the bidisc influences the characterization of distinguished varieties with respect to the open symmetrized bidisc.

Remembering that the first coordinate is the sum and the second coordinate is the product of bidisc elements, a typical point of the symmetrized bidisc will be denoted by $(s,p)$; the same convention will be used to denote a typical point of $\mathbb C^2$, while dealing with the distinguished varieties with respect to this domain.

A substantial refinement of the Pal--Shalit \cite{Pal-Shalit} characterization is obtained. Given a matrix $F$ with $\nu(F)<1$, the set
\begin{align*}
\mathcal W_{F}=\{(s,p)\in\mathbb{C}^2: \det\left(F^*+pF-sI\right)=0\}
\end{align*}
is a distinguished variety with respect to $\mathbb{G}$. Pal and Shalit, by an ingenious application of the concept of the fundamental operator of a $\Gamma$-contraction, proved in Theorem 3.5 of \cite{Pal-Shalit} that given a distinguished variety $\mathcal{W}$ with respect to $\mathbb{G}$, there is a matrix $F$ with $\nu(F)\leq{1}$ such that $\mathcal{W}\cap\mathbb G=\mathcal{W}_{F}\cap\mathbb G$. We improve it. The refinement that we present is as follows.
\begin{thm} \label{ourGammaDV}
Let $(\cF, P, U)$ be a finite model triple such that $\nu(PU+U^*P^\perp)<1$. Let $F=PU+U^*P^\perp$. Then $\mathcal W_{F}$ is a distinguished variety with respect to $\mathbb G$. Conversely, if $\mathcal W$ is a distinguished variety with respect to $\mathbb{G}$, then there is a finite model triple  $(\cF,P,U)$ such that $\mathcal W\cap\mathbb G=\mathcal W_{F}\cap\mathbb G$ with $F=PU+U^*P^\perp$.
\end{thm}

\begin{proof} The forward direction follows from the theorem of Pal and Shalit \cite[Theorem 3.5]{Pal-Shalit}.

For the converse direction, let $\mathcal{W}$ be a distinguished variety with respect to $\mathbb G$, there exists a distinguished variety $\mathcal{V}$ with respect to $\mathbb D^2$ such that $\pi(\mathcal{V})=\mathcal{W}$. Indeed,
$$\mathcal{V} = \{ (z,w) : \pi(z_1,z_2) \in \mathcal{W} \}.$$
We know by Theorem \ref{ourDV} that there exists a finite model triple $(\cF,P,U)$ such that
\begin{align*}
\mathcal{V}\cap\mathbb D^2=\bigcup_{z\in\mathbb D}\sigma_T(P^\perp U+zPU,U^*P+zU^*P^\perp).
\end{align*}Therefore
\begin{align*}
\mathcal{W}\cap\mathbb G=\pi(\mathcal{V}\cap\mathbb D^2)=\bigcup_{z\in\mathbb D}\pi(\sigma_T(P^\perp U+zPU,U^*P+zU^*P^\perp)),
\end{align*}
which, in view of the polynomial spectral mapping theorem, is the same as
\begin{align*}
\mathcal{W}\cap\mathbb G=\bigcup_{z\in\mathbb D}\sigma_T(P^\perp U+zPU +U^*P+zU^*P^\perp, zI_\cF) = \bigcup_{z \in \mathbb D} \sigma_T(F^* + zF , zI_\cF) \end{align*}
where $F = P U + U^* P^\perp$.

Let $(s,p)\in \mathcal{W}\cap\mathbb G$. Then, there exists a non-zero vector $w$ and a $z$ in $\mathbb D$ such that
\begin{align*}
(F^*+ z F)w =sw \text{ and }zw=pw.
\end{align*}
So, $z=p$ and $(F^* + pF - s)w = 0$. So, $ \det(F^*+pF-sI_{\cF})=0$.

Conversely, if $(s,p)$ satisfies $ \det(F^*+pF-sI_{\cF})=0$, then there is a non-zero vector $w$ such that $(F^* + pF - s)w = 0$ showing that
$$(s,p) \in \sigma_T(F^* + pF , pI_\cF) \subset \bigcup_{z \in \mathbb D} \sigma_T(F^* + zF , zI_\cF).$$
This proves the theorem. \end{proof}

Unlike the case of the bidisc, the condition $\nu(PU+U^*P^\perp)<1$ is not necessary as we shall see.

Why is this result a refinement? It is a refinement because while every operator of the form $PU+U^*P^\perp$ has numerical radius no larger than $1$, the converse is not true, i.e., there are $F$ with $\nu(F)\leq{1}$ but $F$ can not written in the form $PU+U^*P^\perp$.

\begin{lemma} \label{FbutnotPU}
If the two eigenvalues $\lambda_1$ and $\lambda_2$ of a given $A\in{M_2(\mathbb{C})}$ satisfy
\begin{align}\label{Condi_mod}
|\lambda_1|\neq{|\lambda_2|},
\end{align}
then $A$ can not be written as $PU+U^*P^\perp$ for any projection $P$ and unitary $U$.

The converse is not true, i.e., equality in \eqref{Condi_mod} does not necessarily imply that $A$ must be of the form $PU+U^*P^\perp$ for some projection $P$ and unitary $U$.
\end{lemma}
\begin{proof}
To prove the first part, suppose $A=PU+U^*P^\perp$ for some projection $P$ and unitary $U$. First note that $P$ cannot be a  trivial projection because otherwise $A$ is a unitary matrix which contradicts \ref{Condi_mod}. Since P is a non-trivial projection, there exists a unitary matrix $U_1$ such that
\begin{align*}
P=U_1^*\begin{bmatrix}
1 & 0\\0 & 0
\end{bmatrix}U_1
\text { and }
P^{\perp}=U_1^*\begin{bmatrix}
0 & 0\\0& 1
\end{bmatrix}U_1.
\end{align*}
So,
\begin{align*}
A=U_1^*\begin{bmatrix}
1 & 0\\0 & 0
\end{bmatrix}U_1U+U^*U_1^*\begin{bmatrix}
0 & 0\\0& 1
\end{bmatrix}U_1.
\end{align*}
From the above equation we have,
\begin{align*}
U_1AU_1^*=\begin{bmatrix}
1 & 0\\0 & 0
\end{bmatrix}U_1UU_1^*+ (U_1UU_1^*)^*\begin{bmatrix}
0 & 0\\0& 1
\end{bmatrix}
\end{align*}
Let the unitary matrix $U_1UU_1^*$ be denoted by $W=[w_{ij}]$. Then
\begin{align}\label{U1andW}
U_1AU_1^*=\begin{bmatrix}
1 & 0\\0 & 0
\end{bmatrix}W+ W^*\begin{bmatrix}
0 & 0\\0& 1
\end{bmatrix} = \begin{bmatrix}
w_{11} & w_{12}+\overline{w_{21}}\\0 & \overline{w_{22}}
\end{bmatrix}
\end{align}
Thus, $w_{11}=\lambda_1$ and $\overline{w}_{22}=\lambda_2$ (or the other way, the treatment of which is the same) which implies that $\lambda_1$ and $\lambda_2$ have to agree in moduli because these are the diagonal entries of a $2 \times 2$ unitary matrix. This contradicts  \eqref{Condi_mod}.

For the converse part, we consider the matrices $A=\sbm{\lambda&1\\0&\lambda}$, $0<|\lambda|\leq 1/2$. Suppose $A=PU+U^*P^\perp$ for some $P$ and $U$. Since $A$ is not a unitary, $P$ cannot be trivial. Similar arguments as above yields that there is a unitary $U_1$ and a unitary $W$ such that \eqref{U1andW} holds.

Since $A$ and $U_1AU_1^*$ have the same eigenvalues and the eigenvalues of an upper-triangular matrix occur along the diagonal, we see that the $2 \times 2$ unitary matrix $W$ has to be of the form $\sbm{ \lambda & x \\ -\bar{x} & \bar{\lambda} }$ for some $x \in \mathbb C$. But this means that $w_{12}+\overline{w_{21}} = 0$ showing that $U_1AU_1^*$ and hence $A$ is a diagonal matrix which is a contradiction.
\end{proof}
It is of independent interest to note that for $A=\sbm{0&1\\0&0}$  there are $2\times2$  choices of $P$ and $U$ such that $A=PU+U^*P^\perp$, viz.,
$$
\begin{bmatrix}
  0&1\\0&0
\end{bmatrix}=\begin{bmatrix}
  1&0\\0&0
\end{bmatrix}\begin{bmatrix}
  0&\zeta\\\zeta&0
\end{bmatrix}+\begin{bmatrix}
  0&\overline{\zeta}\\\overline{\zeta}&0
\end{bmatrix}\begin{bmatrix}
  0&0\\0&1
\end{bmatrix}
$$where $\zeta$ is either $1/2+i\sqrt{3}/2$ or its conjugate. Similarly, if we start with the matrix $\sbm{\lambda&0\\0&\lambda}$ where $\lambda$ comes from the unit disc $\mathbb D$, then we have
$$
\begin{bmatrix}
  \lambda &0\\0&\lambda
\end{bmatrix}=\begin{bmatrix}
  1&0\\0&0
\end{bmatrix}\begin{bmatrix}
  \lambda &(1-|\lambda|^2)^{\frac{1}{2}}\\-(1-|\lambda|^2)^{\frac{1}{2}}&\overline{\lambda}
\end{bmatrix}+\begin{bmatrix}
  \overline{\lambda }&-(1-|\lambda|^2)^{\frac{1}{2}}\\(1-|\lambda|^2)^{\frac{1}{2}}&\lambda
\end{bmatrix}\begin{bmatrix}
  0&0\\0&1
\end{bmatrix}.
$$

\begin{remark}
Starting with a $2\times2$ matrix $A$ with $\nu(A)\leq 1$, it is natural to wonder if there is a $2\times2$ choice of $P$ and $U$ in Theorem \ref{ourGammaDV} for $\cW_A$. Consider for $0<|\lambda|<1$, the matrix $A=\sbm{\lambda&0\\0&0}$. We show that there is no choice of 2-dimensional $P$ and $U$ such that $\cW_A=\cW_F$ with $ F = PU+U^*P^\perp$. It is easy to compute that
$$
\cW_A=\{(\overline{\lambda}+\lambda p,p):p\in\mathbb C\} \cup \{(0,p) : p\in\mathbb C\}.
$$
Suppose on the contrary that there is a $2\times2$ choice of $P$ and $U$ such that $\cW_A=\cW_F$ where $F = PU+U^*P^\perp$. As above, get a unitary $U_1$ such that $P=U_1^*\sbm { 1 & 0 \\ 0 & 0 } U_1$. So, $PU + U^* P^\perp = U_1^* (\sbm{1&0\\0&0}U_1U + U^* U_1^* \sbm{0&0\\0&1}) U_1$. Denote $U_1U$ by $W$ and if $W=[w_{ij}]$, then
$$
F=U_1^* \left( \begin{bmatrix}
1 & 0\\0 & 0
\end{bmatrix}W+W^*\begin{bmatrix}
0 & 0\\0 & 1
\end{bmatrix} \right) U_1= U_1^*  \begin{bmatrix}
  w_{11}&w_{12}+\overline{w_{21}}\\0&\overline{w_{22}}
\end{bmatrix}   U_1.
$$
Since determinant is a unitary invariant, we have
\begin{align}
\notag\det(F^*+pF-sI)&= \det\begin{bmatrix}
  \overline{w_{11}}+ w_{11} p-s& p(w_{12}+\overline{w_{21}})\\ (\overline{w_{12}}+w_{21})& w_{22}+ \overline{w_{22}}p -s\end{bmatrix}\\\label{poly}
&=(\overline{w_{11}}+ w_{11} p-s)(w_{22}+ \overline{w_{22}}p -s)-p|\overline{w_{12}}+w_{21}|^2.
\end{align}If $\cW_{PU+U^*P^\perp}=\cW_A$, then $(0,p)$ would be a solution of $\det(F^*+pF-sI)=0$ for all $p\in\mathbb C$. Thus equating the coefficients we get
$$
\overline{w_{11}} w_{22}=0 \text{ and } w_{11}w_{22}+\overline{w_{11}w_{22}}=|\overline{w_{12}}+w_{21}|^2.
$$
This shows that either $w_{11}=0$ or $w_{22}=0$ and $w_{21}=-\overline{w_{12}}$. Since $WW^*=I$, we get both $w_{11}=0$ and $w_{22}=0$. Putting all this in \eqref{poly} we get $\det(F^*+pF-sI)=s^2$. Note that the points $(\overline{\lambda}+\lambda p,p)$ are in $\cW_A$ but clearly they do not satisfy $s^2=0$. Consequently, $\cW_A \neq \cW_{PU+U^*P^\perp}$ for any $2\times 2$ choice of $P$ and $U$.

It is a matter of easy computation to check that if $P_1=\sbm{1&0\\0&0}$ and for $\lambda\in\mathbb D$, $U_\lambda=\sbm{\lambda &(1-|\lambda|^2)^{\frac{1}{2}}\\-(1-|\lambda|^2)^{\frac{1}{2}}&\overline{\lambda}}$, then
$$
P=\begin{bmatrix}
  P_1&0\\0&P_1
\end{bmatrix} \text{ and }U=\begin{bmatrix}
  U_\lambda &0\\0&U_{0}
\end{bmatrix}
$$ satisfies $\cW_{\sbm{\lambda&0\\0&0}}=\cW_{PU+U^*P^\perp}$. We believe that this is the minimal choice of $P$ and $U$ for $\sbm{\lambda&0\\0&0}$, i.e., there is no $3$-dimensional choice of $P$ and $U$ for $\cW_{\sbm{\lambda&0\\0&0}}$ either.
\end{remark}
We give two examples to show that in the case  $\nu(PU+U^{*}P^{\perp})=1$, nothing can be said about whether the associated variety is distinguished or not.

\begin{example} \label{royal}
Consider the model triple
\begin{align*}
 \cF=\begin{bmatrix} \mathbb C\\\mathbb C\end{bmatrix},
 \quad P=\begin{bmatrix}
     1 & 0\\
     0& 0
    \end{bmatrix}\quad \text{and}\quad
    U=\begin{bmatrix}
     1 & 0\\
     0 & 1
    \end{bmatrix}.
  \end{align*}
Then $\nu(PU+U^{*}P^{\perp})=1$ and the set
\begin{align*}
\cW_{PU+U^*P^\perp}&=\{(s,p)\in\mathbb C^2: \det\left((U^*P+P^\perp U)+p(PU+U^*P^\perp)-sI\right)=0\}\\
&=\{(s,p)\in{\mathbb{C}}^2: \left(1+p-s\right)^2=0\}
\end{align*}
is not a distinguished variety because the point $(1+1/2, 1/2)$ is in $\cW\cap(\partial\Gamma\setminus b\Gamma)$.
\end{example}
\begin{example}\label{E:RV}
Consider the model triple
\begin{align*}
  \cF=\begin{bmatrix} \mathbb C\\\mathbb C\end{bmatrix},\quad P=\begin{bmatrix}
     1 & 0\\
     0& 0
    \end{bmatrix} \quad\text{and}\quad
    U=\begin{bmatrix}
     0 & 1\\
     1 & 0
    \end{bmatrix}.
  \end{align*} Then $\nu(PU+U^{*}P^{\perp})=1$ but nevertheless the set
\begin{align*}
\mathcal W_{PU+U^*P^\perp}&=\{(s,p)\in\mathbb C^2: \det\left((U^*P+P^\perp U)+p(PU+U^*P^\perp)-sI\right)=0\}\\
&=\{(s,p)\in\mathbb{C}^2: s^2-4p=0\}
\end{align*}
is a distinguished variety. This is the so called the {\em royal} variety $\{ (2z, z^2) : z \in \mathbb C\}$ which plays a special role in the understanding of geometry and function theory on the symmetrized bidisc, see \cite{ALY}.
\end{example}
\begin{remark}
The royal variety is a special distinguished variety. It has played a significant role in understanding the complex geometry of the domain, cf.\ \cite{ALY_MAMS}. This, coupled with Example \ref{royal}, makes it natural to wonder whether there is an example of a projection $P$ and a unitary $U$ with $\nu(PU+U^*P^\perp)=1$ such that $\cW_{PU+U^*P^\perp}$ is a distinguished variety but not the royal variety. One can easily find an answer by slightly perturbing Example \ref{E:RV}: set
$$
(\cF,P,U)=\left(\begin{bmatrix} \mathbb C^2\\ \mathbb C^2\end{bmatrix},\begin{bmatrix}
P_1&0\\0&P_1
\end{bmatrix},
\begin{bmatrix}
U&0\\0&iU
\end{bmatrix} \right),
$$where $P_1=\sbm{1&0\\0&0}$ and $U=\sbm{0&1\\1&0}$. It is easy to compute that
$$\cW_{PU+U^*P^\perp}=\{(0,-z^2):z\in\mathbb C\}\cup\{(2z,z^2):z\in\mathbb C\}.$$
Is it possible to find an example of a projection $P$ and a unitary $U$ with $\nu(PU+U^*P^\perp)=1$ such that $\cW_{PU+U^*P^\perp}$ is a distinguished variety that does not contain the royal variety? While it seems plausible, we do not have a concrete example to establish it.
\end{remark}

\section{ One dimensional distinguished varieties in the Polydisc}
The dimension of an algebraic variety in $\mathbb C^d$ is the maximal dimension of tangent spaces at regular points, see page 22 of \cite{GH}.

One dimensional distinguished varieties in the polydisc play an important role for extremal Nevanlinna Pick problems, see Theorem 1.3 in \cite{DS}.
 We need the Berger--Coburn--Lebow theorem in its full generality for characterizing them.

Consider $d $ $(\geq 3)$ commuting isometries $V_1, V_2, \ldots, V_d$ on a Hilbert space $\mathcal{H}$ and the Wold decomposition of the product
$V:=V_1V_2\ldots  V_d$. This means that up to unitary identification $\mathcal{H}=H^2(\cD_{V^*})\oplus \mathcal{H}_{u}$ and
$$
V=\begin{bmatrix}
M_z&0\\ 0& W\end{bmatrix}
:\begin{bmatrix} H^2(\cD_{V^*})\\ \mathcal{H}_{u}\end{bmatrix}\to\begin{bmatrix} H^2(\cD_{V^*})\\ \mathcal{H}_{u}
\end{bmatrix},
$$where $W=V|_{\cH_u}$ is unitary. Berger, Coburn and Lebow \cite[Theorem 3.1]{BCL} proved that for each $j=1,2,\dots,d$, there exist projection operators $P_j$ and unitary operators $U_j$ in $\cB(\cD_{V^*})$, and commuting unitary operators $W_1,W_2,\dots,W_d$ in $\cB(\cH_u)$ such that under the same unitary identification
$$
V_j=\begin{bmatrix} M_{(P_j^\perp+zP_j)U_j}&0\\ 0& W_j\end{bmatrix}:\begin{bmatrix} H^2(\cD_{V^*})\\ \mathcal{H}_{u}\end{bmatrix}\to
\begin{bmatrix} H^2(\cD_{V^*})\\ \mathcal{H}_{u}\end{bmatrix}.
$$

\begin{definition}
A {\em model tuple} is a tuple
$$\chi = (\mathcal{F}, P_1, P_2, \ldots, P_d, U_1, U_2, \ldots, U_d)$$ where $\mathcal{F}$ is a Hilbert space, $P_i$ are projections and $U_i$ are unitary operators in $\cB(\cF)$ such that at least one of the projections $P_i$ is non-trivial and for each $z\in{\mathbb{D}}$
$$\Phi_i(z):= P_i^\perp U_i+zP_iU_i$$are commuting operators. These functions will be called the {\em Berger--Coburn--Lebow (BCL) functions}. If $\mathcal F$ is finite dimensional, then the model tuple is called a finite model tuple.

A model tuple will be called $pure$ if $\Phi_1(z)\Phi_2(z)\ldots\Phi_d(z)=zI \text{ for all } z\in{\mathbb{D}}.$
\end{definition}
When we were discussing the case $d=2$ in the earlier sections, the projection $P_2$ and the unitary $U_2$ were so chosen ($ P_2 = U_1^*P_1^\perp U_1$ and $U_2 = U_1^*$) that this property of being pure was automatic.

For notational convenience, we shall use the notations $\cP:=(P_1, P_2,\dots, P_d)$ and $\cU:=(U_1, U_2, \dots, U_d)$.
Given a finite model tuple $(\cF,\cP,\cU)$, consider the set
 \begin{align}\label{DVPDisc}
 \cW_{\cP, \cU}=\{(z_1,z_2,\dots,z_d)\in\mathbb{C}^d:(z_1,z_2,\dots,z_d)\in\sigma_T(\Phi_1(\underline{z}), \Phi_2(\underline{z}), \ldots ,\Phi_d(\underline{z} ))\}
 \end{align}
where $\underline{z}:=z_1z_2\dots z_d$.

\begin{definition}
An algebraic variety $\cW$ in $\mathbb C^d$ is said to be symmetric if
\begin{align}\label{symm_Poly}
(z_1,z_2,\dots, z_d)\in\cW \text{ if and only if }(\frac{1}{\bar z_1},\frac{1}{\bar z_2},\dots, \frac{1}{\bar z_d})\in\cW_{\cP,\cU}
\end{align}
for non-zero $z_1,z_2,\dots, z_d$.
\end{definition}

\begin{thm}\label{T:thm2}
For a finite pure model tuple $(\cF,\cP,\cU)$, $\cW_{P,U}$ as defined in \eqref{DVPDisc} is a one dimensional symmetric algebraic variety in $\mathbb C^d$.
 Moreover, the following are equivalent:
\begin{itemize}
\item[(i)] $\cW_{\cP,\cU}$ is a distinguished variety with respect to $\mathbb D^d$;
\item[(ii)]  For all $z$ in the open unit disc $\mathbb D$,
\begin{equation} \label{Sec-StrangeConditions1_Poly}
 \nu( P_j^\perp U_j+zP_jU_j) < 1  \text{ for all  } j=1,\dots, d ;
 \end{equation} and
\item[(iii)]
\begin{equation}\label{Variety_Cont}
\cW_{\cP,\cU}\subset \mathbb D^d\cup\mathbb T^d\cup\mathbb E^d.
\end{equation}
\end{itemize} Moreover, $\cW_{\cP,\cU}$ can be written as
\begin{align}\label{BDiskFoli}
\cW_{\cP,\cU}=\bigcup_{z\in\mathbb C}\sigma_T(P_1^\perp U_1+zP_1U_1,\dots, P_d^\perp U_d+zP_dU_d).
\end{align}

Conversely, if $\cW$ is any one dimensional distinguished variety with respect to $\mathbb D^d$, then there exists a finite pure model tuple $(\cF,\cP,\cU)$  such that
$$
\cW\cap\mathbb D^d=\cW_{\cP,\cU}\cap\mathbb D^d.
$$
\end{thm}
\begin{proof}
The proof of the first part of the theorem (including the equivalence of (i), (ii) and (iii)) progresses along the same lines as the two variable situation except for non-emptiness and one dimensionality of $\cW_{\cP, \cU}$. We shall prove only these two.

To see that $\cW_{\cP,\cU}\cap\mathbb D^d$ is non-empty, start with a non-trivial projection $P_j$. Then the matrix $P_j^\perp U_j$ has a zero eigenvalue. Since $(P_1^\perp U_1, P_2^\perp U_2,...., P_d^\perp U_d)$ is a commuting tuple of matrices, there is a joint eigenvalue $(\lambda_1, \ldots , \lambda_{j-1}, 0, \lambda_{j+1}, \lambda_d)$. Because of the containment \eqref{Variety_Cont}, the $\lambda_i$ have to be in $\mathbb D$ for all $i$. Hence $\cW_{\cP,\cU}\cap\mathbb D^d$ is non-empty.

To see that $\cW_{\cP,\cU}$ is one dimensional, let the dimension be $k$. Choose a regular point $(\lambda_1, \lambda_2, \ldots , \lambda_d)$ in $\cW_{\cP,\cU}$. Let $c = \lambda_1 . \lambda_2 . \ldots . \lambda_d$. Consider the hyperplane $Y = \{ (z_1, z_2, \ldots , z_d) : z_1 . z_2 \ldots . z_d = c\}$. Then, by Theorem 8.1 in \cite{Bez}, we get
\begin{align*} \dim (\cW_{\cP,\cU} \cap Y ) & \ge \dim (\cW_{\cP,\cU}) + \dim (Y) - d \\
\text{ or, } 0 & \ge k + d - 1 - d \end{align*}
showing that $k \le 1$. Since the set $\cW_{\cP,\cU}$ is infinite, $k = 1$.

The proof of the converse part, i.e., if $\cW$ is a one dimensional distinguished variety with respect to $\mathbb{D}^d$, then $\cW\cap\mathbb D^d=\cW_{\cP,\cU}\cap\mathbb D^d$ for some pure finite model tuple $(\cF,\cP,\cU)$, requires a few lemmas and a theorem of Scheinker that we list below. We sometimes omit the proofs because they are either elementary or along the same lines as before.

\begin{lemma}\label{Poly-Con}
Let $\cW$ be an algebraic variety in $\mathbb{C}^{d}$. Then $\cW \cap \overline{\mathbb D^d}$ is polynomial convex.
\end{lemma}

As before, we shall denote by $\partial{\cW}$ the set $\cW \cap \mathbb T^d$ for an algebraic variety $\cW$ in $\mathbb C^d$. If $\mu$ is a finite positive measure on $\partial{\cW}$, denote by $H^2(\mu)$ the norm closure of polynomials in $L^2(\partial{\cW},\mu)$. The following theorem is by Scheinker, see Theorem 3.1 in \cite{DS}.

\begin{thm}\label{m:bound}
Given a one dimensional distinguished variety $\cW$ with respect to $\mathbb{D}^d$, there is a finite regular Borel measure $\mu$ on $\partial{\cW}$ such that every point in $\cW\cap\mathbb D^d$ is a bounded point evaluation for $H^2(\mu)$ and such that the span of the evaluation functionals are dense in $H^2(\mu)$.
\end{thm}
Let   $M_{z_1},M_{z_2},\ldots,M_{z_d}$ denote the multiplication by the coordinate functions. Scheinker's result above along with polynomial convexity of $\cW \cap \mathbb D^d$ gives us the following result.

\begin{lemma}\label{p.v}
Let $\cW$ be a one dimensional distinguished variety with respect to $\mathbb{D}^d$. Then a point $(z_1,z_2,\ldots,z_d)$ is in $\cW\cap{\mathbb{D}^d}$ if and only if $(\overline{z_1},\overline{z_2},\ldots,\overline{z_d})$ is a joint eigenvalue of $(M_{z_1}^{*},M_{z_2}^{*},\ldots,M_{z_d}^{*})$.
\end{lemma}

\begin{proof}
Let $w$ be a point in ${\cW}\cap\mathbb D^d$ and $k_w$ be the corresponding evaluation functional on $H^2(\mu)$. By Scheinker's theorem, $H^2(\mu)$ is a reproducing kernel Hilbert space. Hence $M_{f}^{*}k_w=\overline{f(w)}k_w$ for every multiplier $f$ and every kernel function $k_w$. In particular, $(\overline{z_1},\overline{z_2},\ldots,\overline{z_d})$ is a joint eigenvalue of $(M_{z_1}^{*},M_{z_2}^{*},\ldots,M_{z_d}^{*})$.

Conversely, if $(\overline{z_1},\overline{z_2},\ldots,\overline{z_d})$ is a joint eigenvalue of $(M_{z_1}^{*},M_{z_2}^{*},\ldots,M_{z_d}^{*})$, then there is a unit vector $u$ such that $M_{z_i}^{*}u=\overline{z_i}u$, for all $i=1,2,\ldots,d$.
Let $f$ be any polynomial in $d$ variables. Then  $f(z_1,z_2,\ldots,z_d)=\langle u, M_{f}^{*}u\rangle$. Therefore,
\begin{align*}
|f(z_1,z_2,\ldots,z_d)|\leq{||M_{f}||}=\sup_{(z_1,z_2,\ldots,z_d)\in{\cW}}{|f(z_1,z_2,\ldots,z_d)|}.
\end{align*}
So $(z_1,z_2,\ldots,z_d)$ is in the polynomial convex hull of $\cW\cap\mathbb D^d$. Since by Lemma \ref{Poly-Con}, $\cW\cap\mathbb D^d$ is polynomial convex, we have that $(z_1,z_2,\ldots,z_d)$ is in $\cW\cap\mathbb D^d$.
\end{proof}

We now note that the co-ordinate multiplications are pure isometries on $H^2(\mu)$ and it follows from Scheinker's work, Theorem 3.6 in \cite{DS}, that the defect space of the product of these pure isometries is finite dimensional. Thus, the situation is ripe to apply the Berger--Coburn--Lebow theorem mentioned at the beginning of this section. The rest of the proof now is the same as the proof in the two variable situation.

\end{proof}

\begin{remark}
Following arguments of Section 2, it can be shown that $\nu(\Phi_i(z)) < 1$ for all $i=1,2, \ldots , d$ and for all $z \in \mathbb D$ if the real-valued functions
$$z \rightarrow \nu(P_j^\perp U_j+z U_j^*P_j)$$
are all non-constant functions on the open unit disc $\mathbb{D}$.
\end{remark}

\vspace{0.1in} \noindent\textbf{Acknowledgement:}

The first and second named authors' research is supported by the University Grants Commission Centre for Advanced Studies.
The research of the third named author is supported partly by Professor Kalyan B.\ Sinha's Distinguished Fellowship Grant of SERB, Govt.\ of India, and DST-INSPIRE Faculty Fellowship DST/INSPIRE/04/2018/002458. The authors are thankful to Professor B.\ Krishna Das for a brief but valuable conversation.

\end{document}